\documentclass{article}
\usepackage{amssymb,amsmath,amsthm,tikz,multirow,nccrules,geometry}
\usetikzlibrary{calc, arrows, arrows.meta}
\title{Tiling of Hyperbolic Surface by Multiple Tiles}
\author{Chunlin Li, Erxiao Wang\thanks{Corresponding author (wang.eric@zjnu.edu.cn).  Research was supported by National Natural Science Foundation of China NSFC-RGC 12361161603 and Key Projects of Zhejiang Natural Science Foundation LZ22A010003.}, Wu Jie,
Zhejiang Normal University \\
Min Yan\thanks{Research was supported by NSFC-RGC Joint Research Scheme N-HKUST607/23 and Hong Kong RGC General Research Fund 16310925.}, 
Hong Kong University of Science and Technology}

\usepackage[hidelinks, hyperindex]{hyperref}

\hyperref[sec:function]{}

\newcommand{\rr}[1]{\textcolor{red}{#1}}

\newcommand{\mc}{\mathcal}
\newcommand{\bb}{\mathbb}

\makeatletter
\newsavebox\myboxA
\newsavebox\myboxB
\newlength\mylenA

\newcommand*\xar[2][0.7]{%
    \sbox{\myboxA}{$\m@th#2$}%
    \setbox\myboxB\null
    \ht\myboxB=\ht\myboxA%
    \dp\myboxB=\dp\myboxA%
    \wd\myboxB=#1\wd\myboxA
    \sbox\myboxB{$\m@th\overline{\copy\myboxB}$}
    \setlength\mylenA{\the\wd\myboxA}
    \addtolength\mylenA{-\the\wd\myboxB}%
    \ifdim\wd\myboxB<\wd\myboxA%
       \rlap{\hskip 0.5\mylenA\usebox\myboxB}{\usebox\myboxA}%
    \else
        \hskip -0.5\mylenA\rlap{\usebox\myboxA}{\hskip 0.5\mylenA\usebox\myboxB}%
    \fi}
\makeatother

\makeatletter
\newsavebox\myboxC
\newsavebox\myboxD
\newlength\mylenC

\newcommand*\yar[2][0.7]{%
    \sbox{\myboxC}{$\m@th#2$}%
    \setbox\myboxD\null
    \ht\myboxD=\ht\myboxC%
    \dp\myboxD=\dp\myboxC%
    \wd\myboxD=#1\wd\myboxC
    \sbox\myboxD{$\m@th\underline{\copy\myboxD}$}
    \setlength\mylenC{\the\wd\myboxC}
    \addtolength\mylenC{-\the\wd\myboxD}%
    \ifdim\wd\myboxD<\wd\myboxC%
       \rlap{\hskip 0.5\mylenC\usebox\myboxD}{\usebox\myboxC}%
    \else
        \hskip -0.5\mylenC\rlap{\usebox\myboxC}{\hskip 0.5\mylenC\usebox\myboxD}%
    \fi}
\makeatother

\makeatletter
\newcommand{\drawhypgeodesic}[5][]{
  
  \pgfmathsetmacro{\xone}{#2}
  \pgfmathsetmacro{\yone}{#3}
  \pgfmathsetmacro{\xtwo}{#4}
  \pgfmathsetmacro{\ytwo}{#5}
  \pgfmathsetmacro{\denom}{\xone*\ytwo - \xtwo*\yone}

  \ifdim\denom pt < 0.00001pt
    \ifdim\denom pt > -0.00001pt
      \draw[#1] (#2,#3) -- (#4,#5);
      \def\isLine{1}
    \else
      \def\isLine{0}
    \fi
  \else
    \def\isLine{0}
  \fi
  
  \ifnum\isLine=0

    \pgfmathsetmacro{\mone}{1 + \xone*\xone + \yone*\yone}
    \pgfmathsetmacro{\mtwo}{1 + \xtwo*\xtwo + \ytwo*\ytwo}
    \pgfmathsetmacro{\dd}{(\yone*\mtwo - \ytwo*\mone)/\denom}
    \pgfmathsetmacro{\ee}{(\xtwo*\mone - \xone*\mtwo)/\denom}
    \pgfmathsetmacro{\cx}{-\dd/2}
    \pgfmathsetmacro{\cy}{-\ee/2}
    \pgfmathsetmacro{\rr}{sqrt(\cx*\cx + \cy*\cy - 1)}
    \pgfmathsetmacro{\thetaone}{atan2(\yone-\cy, \xone-\cx)}
    \pgfmathsetmacro{\thetatwo}{atan2(\ytwo-\cy, \xtwo-\cx)}   
    \pgfmathsetmacro{\dtheta}{\thetatwo - \thetaone}
    \pgfmathsetmacro{\dtheta}{\dtheta > 180 ? \dtheta - 360 : \dtheta}
    \pgfmathsetmacro{\dtheta}{\dtheta <= -180 ? \dtheta + 360 : \dtheta}
    
    \draw[#1] (#2, #3) arc (\thetaone : \thetaone + \dtheta : \rr);
  \fi
}
\makeatother

\newtheorem{theorem}{Theorem}

\newtheorem{proposition}[theorem]{Proposition}
\newtheorem*{theorem*}{Theorem}

\theoremstyle{definition}
\newtheorem*{definition*}{Definition}
\newtheorem*{case*}{Case}
\newtheorem*{subcase*}{Subcase}

\theoremstyle{remark}

\numberwithin{equation}{section}

\begin{document}

\maketitle

\begin{abstract}
Tilings of a surface of negative Euler characteristic by $n$-gons with $n\ge 7$ is a finite problem. We develop the algorithm for finding all the tilings for fixed number of tiles and present the calculation for tilings of surfaces of small genus by two tiles. We also discuss the number of distinct edge lengths in multiple tile tilings.
\end{abstract}

\section{Introduction}

Consider an edge-to-edge tiling of a closed and connected surface of Euler number $\chi<0$ by congruent $n$-gons, such that $n\ge 7$ and all vertices have degree $\ge 3$. In \cite{lwwy1}, we showed that the following
\[
\frac{-2\chi}{n-2}<f\le \frac{-6\chi}{n-6};
\quad
n\le\begin{cases}
3(2-\chi), &\text{odd }n \\
6(1-\chi), &\text{even }n 
\end{cases}.
\]
Therefore for fixed $\chi$, finding all tilings is a finite problem. 

In \cite{lwwy1}, we developed the algorithm for single tile tilings, i.e., with $f=1$. We found all the single tile tilings for surfaces of small genus, and also discussed the geometric existence of such tilings. For the special case of single tile tilings of $2{\bb T}^2$, we compared our findings with the works of Zamorzaeva-Orleanschi \cite{zamo1,zamo2}, and the corrected the mistakes in the earlier papers.

In this paper, we extend the algorithm to more than one tiles. Then we specialize the algorithm to two tile tilings, i.e., with $f=2$. We find all two tile tilings for surfaces of small genus, and explicitly present some of these tilings. We find some new phenomenon that did not occur to single tile tilings, including some extra geometrical condition built into the algorithm, and the various number of distinct edge lengths.

A single tile tiling is the same as a planar diagram, which is a pairing of edges of an $n$-gon. Each pairing is either opposing or twisted. A multiple tile tiling is similarly given by a multiple planar diagram, which is a pairing of edges of several $n$-gons. In \cite{lwwy1}, we showed that a planar diagram can be either encoded by the pairing of edges, or by the corner combinations at vertices in the tiling. We provided the formulae for the conversion between the two codes, which shows the equivalence of the two methods. In this paper.

The geometrical conditions from the pairings (same edge length) and vertices (angle sum) are underdetermined for single tile tilings. In principle, different edge pairs can have distinct edge lengths. Therefore we expect all combinatorially valid tilings to be geometrically realisable. However, the geometrical conditions inherent in multiple tile tilings can sometimes be overdetermined. For example, the number of distinct edge lengths can be as small as 1. Moreover, it is possible that the angle sum equalities do not have positive angle solutions. Therefore we add the condition of the existence of positive angle solution to our program. This filters out a big number of geometrically impossible tilings. The number of tilings in this paper is the number that passes this filter.

A tiling of a surface by a single $n$-gon can have as many as $\frac{n}{2}$ distinct edge lengths. In a tiling of a surface by several congruent $n$-gons, it is possible to have all $n$ edges having distinct edge lengths. In the final section, we give necessary condition for this to happen. Moreover, for two tile tilings, we give necessary and sufficient condition for a surface to have a tiling by two congruent $n$-gons, such that all $n$ edge lengths are distinct.

\section{Tiling of Surface by Congruent Tiles}
\label{tiling}

An edge-to-edge tiling ${\mc T}$ of a surface consists of polygons $T_1,T_2,\dots,T_f$ (called {\em tiles}) of $n_1,n_2,\dots,n_f$ sides. Combinatorially, the tiling is a {\em multiple planar diagram} $D$ that pairs the edges of these polygons. The pairing implies that $n_1+n_2+\dots+n_f=2m$ is an even number. 

Similar to the usual planar diagrams in \cite{lwwy1, zamo1, zamo2}, two edges in a multiple planar diagram can be paired in two ways, corresponding to the two ways the end points are identified. A multiple planar diagram $D$ specifies one of the two ways for each edge pair. As long as we glue all the edge pairs together in ways specified in $D$, we get a surface $S_D$.

We also require the surface $S_D$ to be connected. This means that, any two tiles $T_p$ and $T_q$ are related by a sequence of tiles in between, such that the adjacent tiles have edges that are paired in $D$. A multiple planar diagram with this property is a {\em connected multiple planar diagram}.

In this paper, we assume all tiles are congruent to a prototile $n$-gon $P$. Then $nf=2m$ is an even number. 

In Figure \ref{prototile}, we label the corners of $P$ circularly by $i\in {\bb Z}_n$. We also denote the edge connecting $i$ to $i+1$ by $\bar{i}$. The label gives an orientation of $P$. The congruence to $P$ translates the labels and the orientation to all the tiles. The corners and edges of $T_p$ are $i_p$ and $\bar{i}_p$, for all $i\in {\bb Z}_n$. 

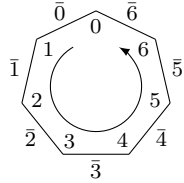
\begin{figure}[htp]
\centering
\begin{tikzpicture}[>=latex,scale=1]

\foreach \a in {0,...,6}
{
\draw
	(38.57+51.43*\a:1) -- (90+51.43*\a:1);

\node at (90+51.43*\a:0.8) {\footnotesize \a};
\node at (115.7+51.43*\a:1.1) {\footnotesize $\bar{\a}$};

}

\draw[->]
	(120:0.6) arc (120:420:0.6);
		
\end{tikzpicture}
\caption{Corners and edges of a prototile heptagon.} 
\label{prototile}
\end{figure}

A multiple planar diagram $D$ is an unordered pairing in the set of all edges of all tiles
\begin{equation}\label{edges}
E=\{\bar{i}_p\colon i_p\in {\bb Z}_n, \; p=1,2\dots,f\}
={\bb Z}_n\times \{1,2,\dots,f\}.
\end{equation}
Moreover, we add a sign $\sigma=\pm 1$ to each (unordered) edge pair $\bar{i}_p\bar{j}_q$ to indicate one of the two ways the edges are glued together. Figure \ref{dtov} shows the meaning of the sign:
\begin{itemize}
\item $\sigma=+$: The {\em opposing} pair on the left of Figure \ref{dtov}. This means the tiles $T_p$ and $T_q$ have the same orientation with respect to the edge pair. 
\item $\sigma=-$: The {\em twisted} pair on the right of Figure \ref{dtov}. This means the tiles $T_p$ and $T_q$ have different orientation with respect to the edge pair. 
\end{itemize}
The lower row is the flip of the upper row. They represent the same glueing of a pair of edges.
 
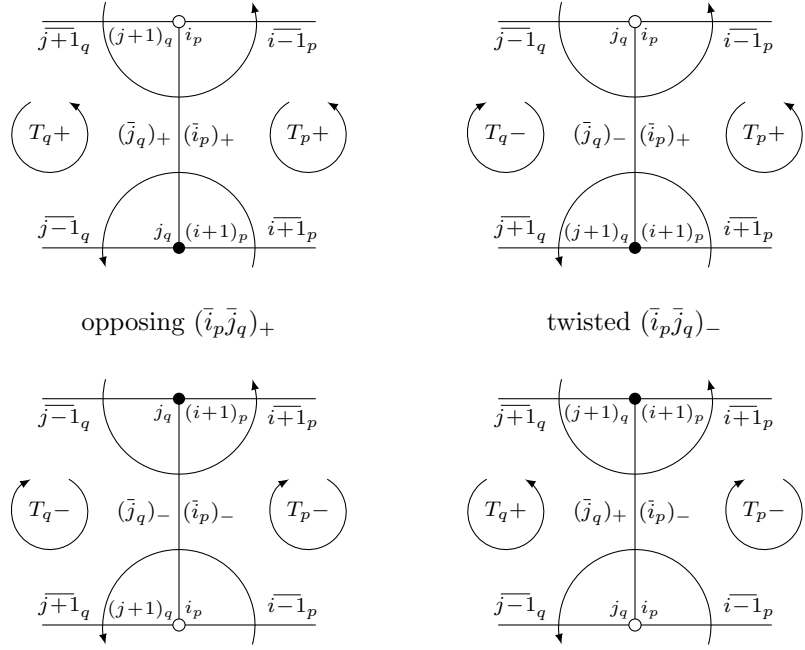
\begin{figure}[htp]
\centering
\begin{tikzpicture}[>=latex,scale=1]

\foreach \a in {1,-1}
\foreach \b in {1,-1}
{
\begin{scope}[shift={(3*\a cm, 2.5*\b cm)}, yscale=\b]

\draw
	(0,-1.5) -- (0,1.5)
	(-1.8,-1.5) -- (1.8,-1.5)
	(-1.8,1.5) -- (1.8,1.5);

\fill 
	(0,-1.5) circle (0.08);
\filldraw[fill=white]
	(0,1.5) circle (0.08);

\end{scope}

\begin{scope}[shift={(3*\a cm, 2.5*\b cm)}]

\draw[yshift=1.5 cm, <-]
	(15:1) arc (15:-195:1);
	
\draw[yshift=-1.5 cm, ->]
	(-15:1) arc (-15:195:1);

\end{scope}	
}

\foreach \a in {1,-1}
{
\begin{scope}[shift={(3*\a cm, 2.5 cm)}]

\draw[xshift=1.7 cm, ->]
	(120:0.5) arc (120:420:0.5);

\node at (0.4,0) {\footnotesize $(\xar{i}_p)_+$};
\node at (1.7,0) {\footnotesize $T_p+$};

\node at (1.5,-1.25) {\footnotesize $\xar{i\!+\!1}_p$};
\node at (1.5,1.25) {\footnotesize $\xar{i\!-\!1}_p$};

\node at (0.2,1.3) {\scriptsize $i_p$};
\node at (0.5,-1.3) {\scriptsize $(i\!+\!1)_p$};

\end{scope}

\begin{scope}[shift={(3*\a cm, -2.5 cm)}]
	
\draw[xshift=1.7 cm, <-]
	(120:0.5) arc (120:420:0.5);

\node at (0.4,0) {\footnotesize $(\xar{i}_p)_-$};
\node at (1.7,0) {\footnotesize $T_p-$};

\node at (1.5,1.25) {\footnotesize $\xar{i\!+\!1}_p$};
\node at (1.5,-1.25) {\footnotesize $\xar{i\!-\!1}_p$};

\node at (0.2,-1.3) {\scriptsize $i_p$};
\node at (0.5,1.3) {\scriptsize $(i\!+\!1)_p$};

\end{scope}
}

\foreach \a in {1,-1}
{
\begin{scope}[shift={(3*\a cm, -2.5*\a cm)}]

\draw[xshift=-1.7 cm, ->]
	(120:0.5) arc (120:420:0.5);
	
\node at (-1.7,0) {\footnotesize $T_q+$};
\node at (-0.45,0) {\footnotesize $(\xar{j}_q)_+$};

\node at (-1.5,-1.25) {\footnotesize $\xar{j\!-\!1}_q$};
\node at (-1.5,1.25) {\footnotesize $\xar{j\!+\!1}_q$};

\node at (-0.2,-1.3) {\scriptsize $j_q$};
\node at (-0.5,1.3) {\scriptsize $(j\!+\!1)_q$};

\end{scope}
}

\foreach \a in {1,-1}
{
\begin{scope}[shift={(3*\a cm, 2.5*\a cm)}]

\draw[xshift=-1.7 cm, <-]
	(120:0.5) arc (120:420:0.5);
	
\node at (-1.7,0) {\footnotesize $T_q-$};
\node at (-0.45,0) {\footnotesize $(\xar{j}_q)_-$};

\node at (-1.5,1.25) {\footnotesize $\xar{j\!-\!1}_q$};
\node at (-1.5,-1.25) {\footnotesize $\xar{j\!+\!1}_q$};

\node at (-0.2,1.3) {\scriptsize $j_q$};
\node at (-0.5,-1.3) {\scriptsize $(j\!+\!1)_q$};

\end{scope}
}

\node at (-3,0) {opposing $(\bar{i}_p\bar{j}_q)_+$};
\node at (3,0) {twisted $(\bar{i}_p\bar{j}_q)_-$};

\end{tikzpicture}
\caption{Edge pair implies adjacent corners at vertices.} 
\label{dtov}
\end{figure}

An edge pair implies adjacent corners at two vertices $\bullet$ and $\circ$. The upper left shows that an opposing pair $(\bar{i}_p\bar{j}_q)_+$ first splits into two ordered pairs $(\bar{i}_p)_+(\bar{j}_q)_+ $ and $(\bar{j}_q)_+(\bar{i}_p)_+$, and then imply adjacent corners 
\begin{equation}\label{dtov_eq1+}
(\bar{i}_p\bar{j}_q)_+
\implies
\begin{cases}
(\bar{i}_p)_+(\bar{j}_q)_+ 
&\implies \bullet=((i+1)_p)_+(j_q)_+\cdots \\
(\bar{j}_q)_+(\bar{i}_p)_+
&\implies \circ=((j+1)_q)_+(i_p)_+\cdots
\end{cases}.
\end{equation}
Here we add the subscript $+$ to all the edges and corners to keep track of the counterclockwise orientations of the two tiles. The pairing is equivalently described by the lower left, which shows that $(\bar{i}_p\bar{j}_q)_+$ can also split into two ordered pairs $(\bar{j}_q)_-(\bar{i}_p)_-$ and $(\bar{i}_p)_-(\bar{j}_q)_-$, and then equivalently imply adjacent corners 
\begin{equation}\label{dtov_eq1-}
(\bar{i}_p\bar{j}_q)_+
\implies
\begin{cases}
(\bar{j}_q)_-(\bar{i}_p)_-
&\implies \bullet=(j_q)_-((i+1)_p)_-\cdots \\
(\bar{i}_p)_-(\bar{j}_q)_-
&\implies \circ=(i_p)_-((j+1)_q)_-\cdots 
\end{cases}.
\end{equation}

The right of Figure \ref{dtov} shows two ways of drawing a twisted pair $(\bar{i}_p\bar{j}_q)_-$. The pair can split into $(\bar{i}_p)_+(\bar{j}_q)_-$, $(\bar{j}_q)_-(\bar{i}_p)_+$, $(\bar{j}_q)_+(\bar{i}_p)_-$, $(\bar{i}_p)_-(\bar{j}_q)_+$, and then imply adjacent corners
\begin{equation}\label{dtov_eq2}
(\bar{i}_p\bar{j}_q)_-
\implies
\begin{cases}
(\bar{i}_p)_+(\bar{j}_q)_-
&\implies \bullet=((i+1)_p)_+((j+1)_q)_-\cdots \\
(\bar{j}_q)_-(\bar{i}_p)_+
&\implies \circ=(j_q)_-(i_p)_+\cdots \\
(\bar{j}_q)_+(\bar{i}_p)_-
&\implies \bullet=((j+1)_q)_+((i+1)_p)_-\cdots \\
(\bar{i}_p)_-(\bar{j}_q)_+
&\implies \circ=(i_p)_-(j_q)_+\cdots
\end{cases}. 
\end{equation}

The formulae \eqref{dtov_eq1+}, \eqref{dtov_eq1-}, \eqref{dtov_eq2} are illustrated in Figure \ref{p2v}. The edge pairs along the boundaries are connected by chords. The solid chords mean opposing pairs, and the dashed chords mean twisted pairs. We have vertices $\bullet$ and $\circ$, with decorations by $\sigma=\pm$. The gray arrows describe the adjacent corners at vertices induced by the edge pairs.

\begin{figure}[htp]
\centering
\begin{tikzpicture}[>=latex,scale=1]


\begin{scope}[shift={(-6,0)}]

\draw
	(50:1.2) -- (80:1.2)
	(150:1.2) -- (180:1.2)
	(65:1.16) to[out=245, in=-15] (165:1.16);

\draw[->]
	(120:0.2) arc (120:420:0.2);
	
\draw[dashed]
	(80:1.2) arc (80:150:1.2)
	(180:1.2) arc (180:410:1.2);

\draw[gray!70, very thick, <-]
	(50:1.1) to[out=180+40, in=5] (180:1.05);

\draw[gray!70, very thick, <-]
	(80:1.1) to[out=180+70, in=180+160] (150:1.05);

\node at (0,0) {\scriptsize $T_p$}; 
	
\end{scope}


\begin{scope}[shift={(-6,-3)}]

\draw
	(50:1.2) -- (80:1.2)
	(150:1.2) -- (180:1.2);

\draw[dashed]
	(80:1.2) arc (80:150:1.2)
	(180:1.2) arc (180:410:1.2);

\draw[dash pattern=on 2pt off 1pt]
	(65:1.16) to[out=245, in=-15] (165:1.16);

\draw[gray!70, very thick, <-]
	(50:1.1) to[out=180+50, in=180+150] (150:1.05);

\draw[gray!70, very thick, ->]
	(180:1.05) to[out=0, in=180+80] (80:1.1);

\draw[->]
	(120:0.2) arc (120:420:0.2);

\node at (0,0) {\scriptsize $T_p$}; 
	
\end{scope}

\foreach \a in {-1,1}
\foreach \y in {0,1}
{
\begin{scope}[shift={(-2*\a, -3*\y)}, scale=\a]

\draw
	(20:1.2) -- (50:1.2)
	(-20:1.2) -- (-50:1.2);
	
\draw[dashed]
	(-20:1.2) arc (-20:20:1.2)
	(50:1.2) arc (50:310:1.2);
	
\end{scope}
}

\foreach \y in {0,1}
\node at (-2,-3*\y) {\scriptsize $T_p$}; 

\foreach \y in {0,1}
\node at (2,-3*\y) {\scriptsize $T_q$}; 


\draw
	(-1.05,0.66) to[out=35, in=145] (1.05,0.66);

\draw[dash pattern=on 2pt off 1pt]
	(-1.05,-0.66) to[out=-35, in=-145] (1.05,-0.66);
	
\begin{scope}[gray!70, very thick]

\draw[->]
	(-0.75,0.5) to[out=30, in=150] (0.8,0.5);

\draw[->]
	(-1.05,0.95) to[out=30, in=150] (1.1,0.95);

\draw[->]
	(-0.75,-0.5) to[out=-40, in=200] (1.1,-0.95);

\draw[->]
	(-1.05,-0.95) to[out=-20, in=220] (0.8,-0.5);

\end{scope}
			
\foreach \x in {1,-1}
\draw[shift={(2*\x,0)}, ->]
	(120:0.2) arc (120:420:0.2);


\begin{scope}[yshift=-3cm]

\draw
	(-1.05,0.66) to[out=35, in=145] (1.05,0.66);

\draw[dash pattern=on 2pt off 1pt]
	(-1.05,-0.66) to[out=-35, in=-145] (1.05,-0.66);

\begin{scope}[gray!70, very thick]

\draw[->]
	(-0.75,-0.5) to[out=-30, in=-150] (0.8,-0.5);

\draw[->]
	(-1.05,-0.95) to[out=-30, in=-150] (1.1,-0.95);

\draw[->]
	(-0.75,0.5) to[out=40, in=160] (1.1,0.95);

\draw[->]
	(-1.05,0.95) to[out=20, in=140] (0.8,0.5);
	
\end{scope}	

\draw[shift={(-2,0)}, ->]
	(120:0.2) arc (120:420:0.2);

\draw[shift={(2,0)}, <-]
	(120:0.2) arc (120:420:0.2);
	
\end{scope}


\foreach \a/\x/\y in 
	{180/-1/0, 50/-1/0, 180/-1/1, -20/0/0, -20/0/1}
{
\begin{scope}[shift={(-2+4*\x, -3*\y)}, shift={(\a:1.2)}]

\filldraw[fill=white]
	(0,0) circle (0.1);

\draw
	(-0.1,0) -- (0.1,0)
	(0,-0.1) -- (0,0.1);
	
\end{scope}
} 


\foreach \a/\x/\y in 
	{80/-1/1, 50/0/0, 130/1/0, 230/1/0, 50/0/1, 160/1/1, 200/1/1}
{
\begin{scope}[shift={(-2+4*\x, -3*\y)}, shift={(\a:1.2)}]

\filldraw[fill=white]
	(0,0) circle (0.1);

\draw
	(-0.1,0) -- (0.1,0);
	
\end{scope}
} 


\foreach \a/\x/\y in 
	{50/-1/1, 20/0/0, 160/1/0, 200/1/0, 20/0/1, 130/1/1, 230/1/1}
{
\begin{scope}[shift={(-2+4*\x, -3*\y)}, shift={(\a:1.2)}]

\fill
	(0,0) circle (0.1);

\draw[white]
	(-0.1,0) -- (0.1,0)
	(0,-0.1) -- (0,0.1);
	
\end{scope}
} 


\foreach \a/\x/\y in {150/-1/0, 80/-1/0, 150/-1/1, -50/0/0, -50/0/1}
{
\begin{scope}[shift={(-2+4*\x, -3*\y)}, shift={(\a:1.2)}]

\fill
	(0,0) circle (0.1);

\draw[white]
	(-0.1,0) -- (0.1,0);
	
\end{scope}
} 
	
\end{tikzpicture}
\caption{Planar diagram to vertices.} 
\label{p2v}
\end{figure}
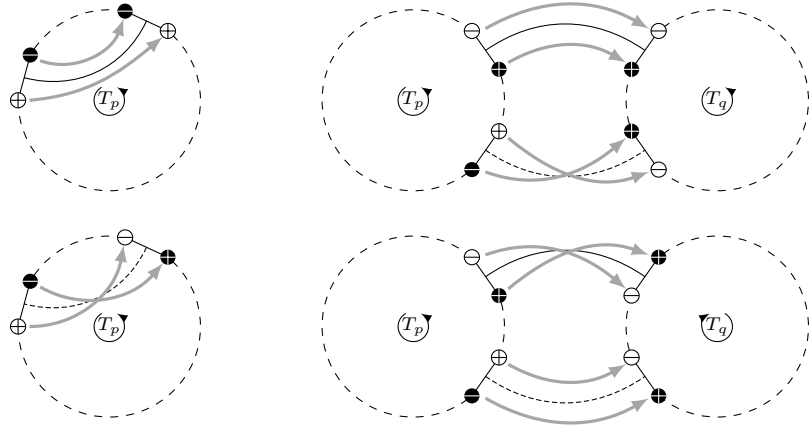

The left two pictures describe the edge pairs in the same tile $T_p$. The right two pictures describe the edge pairs between different tiles $T_p$ and $T_q$. In the upper right, the two tiles have the same counterclockwise orientation. In the lower right, the two tiles have different orientations.

The adjacent corners at vertices combine to give vertices. A vertex is  a {\em circularly ordered subset} of all corners of all tiles
\begin{equation}\label{corners}
C=\{i_p\in {\bb Z}_n, \; p=1,2\dots,f\}
={\bb Z}_n\times \{1,2,\dots,f\},
\end{equation}
with each corner decorated by $\sigma=\pm$. By circularly ordered, we mean that rotation of the circular order and preserving all $\sigma$, or the reversion of the circular order and changing all $\sigma$ to $-\sigma$, represent the same vertex. For example, we have
\[
(0_1)_+(5_2)_-(6_1)_-(4_2)_+(4_1)_+
\overset{\text{rotation}}{=\joinrel=\joinrel=\joinrel=\joinrel=}
(6_1)_-(4_2)_+(4_1)_+(0_1)_+(5_2)_- 
\overset{\text{reversion}}{=\joinrel=\joinrel=\joinrel=\joinrel=}
(5_2)_+(0_1)_-(4_1)_-(4_2)_-(6_1)_+.
\]

We may summarise the formulae \eqref{dtov_eq1+}, \eqref{dtov_eq1-}, \eqref{dtov_eq2} from edge pairs to adjacent corners as the following
\begin{align*}
(i_p)_+\cdots 
&\overset{\scriptsize (\overline{i-1}_p)_+(\bar{j}_q)_{\sigma}}{=\joinrel=\joinrel=\joinrel=\joinrel=\joinrel=\joinrel=}
\begin{cases}
(i_p)_+(j_q)_+\cdots, &
\text{if }\sigma=+ \\
(i_p)_+((j+1)_q)_-\cdots, &
\text{if }\sigma=-
\end{cases},  \\
(i_p)_-\cdots 
&\overset{\scriptsize (\bar{i}_p)_-(\bar{j}_q)_{\sigma}}{=\joinrel=\joinrel=\joinrel=\joinrel=\joinrel=\joinrel=}
\begin{cases}
(i_p)_-(j_q)_+\cdots, &
\text{if }\sigma=+ \\
(i_p)_-((j+1)_q)_-\cdots, &
\text{if }\sigma=-
\end{cases}. 
\end{align*}
For example, the double planar diagram illustrated in Figure \ref{eg1}
\begin{equation}\label{dpd1}
D=((\bar{0}_1\bar{3}_1)_+,
(\bar{6}_2\bar{2}_2)_+,
(\bar{2}_1\bar{0}_2)_-,
(\bar{4}_1\bar{3}_2)_+,
(\bar{6}_1\bar{4}_2)_-,
(\bar{1}_1\bar{1}_2)_-,
(\bar{5}_1\bar{5}_2)_+)
\end{equation}
implies vertices
\begin{align*}
(0_1)_+\cdots
&\overset{(\bar{6}_1)_+(\bar{4}_2)_-}{=\joinrel=\joinrel=\joinrel=\joinrel=}
(0_1)_+(5_2)_-\cdots
\overset{(\bar{5}_2)_-(\bar{5}_1)_-}{=\joinrel=\joinrel=\joinrel=\joinrel=}
(0_1)_+(5_2)_-(6_1)_-\cdots 
\overset{(\bar{6}_1)_-(\bar{4}_2)_+}{=\joinrel=\joinrel=\joinrel=\joinrel=}
(0_1)_+(5_2)_-(6_1)_-(4_2)_+\cdots \\
&\overset{(\bar{3}_2)_+(\bar{4}_1)_+}{=\joinrel=\joinrel=\joinrel=\joinrel=}
(0_1)_+(5_2)_-(6_1)_-(4_2)_+(4_1)_+\cdots 
\overset{(\bar{3}_1)_+(\bar{0}_1)_+}{=\joinrel=\joinrel=\joinrel=\joinrel=}
(0_1)_+(5_2)_-(6_1)_-(4_2)_+(4_1)_+, \\
(0_2)_+\cdots
&\overset{(\bar{6}_2)_+(\bar{2}_2)_+}{=\joinrel=\joinrel=\joinrel=\joinrel=}
(0_2)_+(2_2)_+\cdots
\overset{(\bar{1}_2)_+(\bar{1}_1)_-}{=\joinrel=\joinrel=\joinrel=\joinrel=}
(0_2)_+(2_2)_+(2_1)_-\cdots 
\overset{(\bar{2}_1)_-(\bar{0}_2)_+}{=\joinrel=\joinrel=\joinrel=\joinrel=}
(0_2)_+(2_2)_+(2_1)_-.
\end{align*}
We similarly get the other two vertices $(1_1)_+(3_1)_+(1_2)_-$, $(3_2)_+(6_2)_+(5_1)_+$, and the vertex set
\[
V =((0_1)_+(5_2)_-(6_1)_-(4_2)_+(4_1)_+,\; 
(0_2)_+(2_2)_+(2_1)_-,
(1_1)_+(3_1)_+(1_2)_-,\; (3_2)_+(6_2)_+(5_1)_+).
\]
Therefore the surface has Euler characteristic $4-7+2=-1$. This implies the surface is $S_D=3{\bb P}^2$. 

Figure \ref{eg1} shows two equivalent ways of drawing the double planar diagram \eqref{dpd1}. The solid chords indicate opposing edge pairs, and the dashed chords indicate twisted edge pairs. On the left, both tiles have the same orientation. On the right, the two tiles have different orientation. The directed thick gray paths are obtained by applying the rule in Figure \ref{p2v}, and represent one vertex $(0_1)_+(5_2)_-(6_1)_-(4_2)_+(4_1)_+$ of the tiling.

\vspace{-1cm}

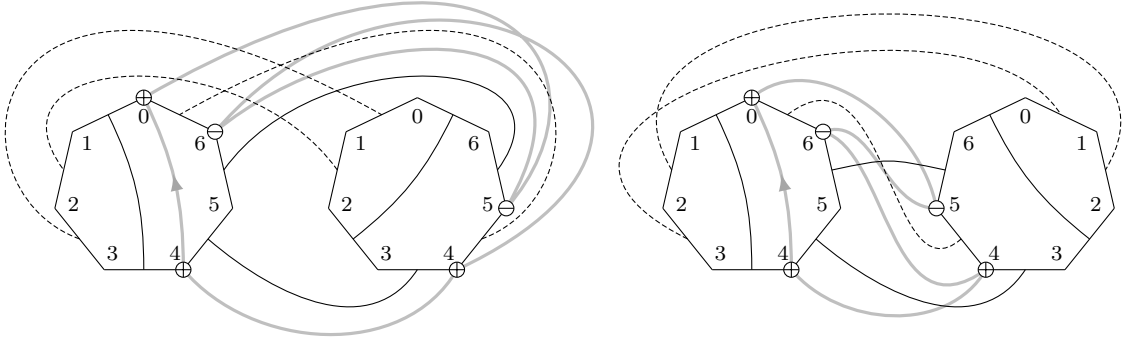
\begin{figure}[htp]
\centering
\begin{tikzpicture}[>=latex,scale=1]

\foreach \a in {0,...,6}
\foreach \u in {1,-1}
\foreach \x in {0,1}
\draw[shift={(1.8*\u+8*\x, 0)}]
	(38.57+51.43*\a:1.2) -- (90+51.43*\a:1.2);

\foreach \a in {0,...,6}
\foreach \x in {-1.8,1.8,6.2}
\node[xshift=\x cm] at (90+51.43*\a:0.95) {\footnotesize \a};

\foreach \a in {0,...,6}
\node[xshift=9.8 cm] at (90-51.43*\a:0.95) {\footnotesize \a};

\foreach \a in {0,...,6}
\foreach \x in {0,1}
{

\coordinate (A\a\x) at ([shift={(-1.8cm+8*\x cm, 0 cm)}] 115.71+51.43*\a:1.08);
\coordinate (B\a\x) at ([shift={(1.8cm+8*\x cm, 0 cm)}] 64.28-51.43*\a:1.08);

\coordinate (X\a\x) at
	([shift={(-1.8cm+8*\x cm, 0 cm)}] 90+51.43*\a:1.2);
\coordinate (Y\a\x) at
	([shift={(1.8cm+8*\x cm, 0 cm)}] 90+51.43*\a:1.2);
}


\draw[gray!50, very thick]
	(X00) to[out=-64.28, in=90] 
	(X40) to[out=-50, in=-120] 
	(Y40) .. controls (6.7,1) and (2,4.1) ..
	(X60) .. controls (1,2.3) and (4.5,2.4) .. 
	(Y50) .. controls (4.5,1.5) and (3,4) .. 
	(X00);

\draw[gray!70, very thick, xshift=-1.8 cm, ->]
	(12.85:0.42) -- ++(102.85:0.1);

\foreach \a/\b/\x in {0/3/0, 0/3/1}
\draw
	(A\a\x) to[out=295.71+51.43*\a, in=295.71+51.43*\b] (A\b\x);
	
\foreach \a/\b/\x in {0/4/0, 2/6/1}
\draw
	(B\a\x) to[out=244.28-51.43*\a, in=244.28-51.43*\b] (B\b\x);

\draw
	(A50) .. controls (0.5,2) and (4,1.8) .. (B10)
	(A40) .. controls (0,-1.5) and (1.2,-2) .. (B30);
	
\draw[dash pattern=on 2pt off 1pt]
	(B50) .. controls (-0.5,2) and (-4,1.8) .. (A10)
	(A20) .. controls (-4.5,0) and (-4,4) .. (B60)
	(B20) .. controls (4.5,0) and (4,4) .. (A60);
	
\foreach \a/\x in {0/-1, 4/-1, 4/1}
{
\begin{scope}[xshift=1.8*\x cm, shift={(90+51.43*\a:1.2)}]

\filldraw[fill=white]
	(0,0) circle (0.1);

\draw
	(-0.1,0) -- (0.1,0)
	(0,-0.1) -- (0,0.1);
	
\end{scope}
} 

\foreach \a/\x in {5/1, 6/-1}
{
\begin{scope}[xshift=1.8*\x cm, shift={(90+51.43*\a:1.2)}]

\filldraw[fill=white]
	(0,0) circle (0.1);

\draw
	(-0.1,0) -- (0.1,0);
	
\end{scope}
} 


\begin{scope}[xshift=8cm]

\draw[gray!50, very thick]
	(X01) to[out=-64.28, in=90] 
	(X41) to[out=-50, in=-120] 
	(Y31) .. controls (0,-2) and (0,0.5) ..
	(X61) to[out=20, in=180] 
	(Y21) to[out=100, in=40] 
	(X01);

\draw[gray!70, very thick, xshift=-1.8 cm, ->]
	(12.85:0.42) -- ++(102.85:0.1);

\draw
	(A51) .. controls (0,0.4) .. (B51)
	(A41) .. controls (0,-1.5) and (1.2,-2) .. (B31);
	
\draw[dash pattern=on 2pt off 1pt]
	(A11) .. controls (-4.5,3) and (4.5,3) .. (B11)
	(A21) .. controls (-6,1) and (1,3) .. (B01)
	(A61) .. controls (0,2) and (0,-1.5) .. (B41);
	
\foreach \a/\x in {0/-1, 4/-1, 3/1}
{
\begin{scope}[xshift=1.8*\x cm, shift={(90+51.43*\a:1.2)}]

\filldraw[fill=white]
	(0,0) circle (0.1);

\draw
	(-0.1,0) -- (0.1,0)
	(0,-0.1) -- (0,0.1);
	
\end{scope}
} 

\foreach \a/\x in {2/1, 6/-1}
{
\begin{scope}[xshift=1.8*\x cm, shift={(90+51.43*\a:1.2)}]

\filldraw[fill=white]
	(0,0) circle (0.1);

\draw
	(-0.1,0) -- (0.1,0);
	
\end{scope}
} 

\end{scope}

\end{tikzpicture}
\caption{One vertex of the double planar diagram \eqref{dpd1}.} 
\label{eg1}
\end{figure}

If the Euler number is even, then we need to further check the orientability to determine the surface. In fact, we will discuss the simplified algorithm for tilings of orientable surfaces in the subsequent Section \ref{orientable}. For the double planar diagram in \eqref{dpd1}, we only mention that the different $\pm$ decorations for $(\bar{1}_1\bar{1}_2)_-$ and $(\bar{5}_1\bar{5}_2)_+$ imply that the surface is not orientable. 

We saw a multiple planar diagram $D$ induces a {\em vertex set}. The vertex set is a partition of all corners $C$ in \eqref{corners} into a disjoint union of circularly ordered subsets. Conversely, we may recover the multiple planar diagram from the vertex set, by reversing the formulae
\begin{align*}
(i_p)_+(j_q)_+\cdots &\implies (\overline{i-1}_p\bar{j}_q)_+, &
(i_p)_-(j_q)_+\cdots &\implies (\bar{i}_p\bar{j}_q)_-, \\
(i_p)_+(j_q)_-\cdots &\implies (\overline{i-1}_p\overline{j-1}_q)_-, &
(i_p)_-(j_q)_-\cdots &\implies (\bar{i}_p\overline{j-1}_q)_+.
\end{align*}
For example, the vertex $(0_1)_+(5_2)_-(6_1)_-(4_2)_+(4_1)_+$ implies the edge pairs
\begin{align*}
(0_1)_+(5_2)_-
&\implies (\bar{6}_1\bar{4}_2)_- &
(5_2)_-(6_1)_- 
&\implies (\bar{5}_2\bar{5}_1)_+=(\bar{5}_1\bar{5}_2)_+, \\
(6_1)_-(4_2)_+ 
&\implies (\bar{6}_1\bar{4}_2)_-, & 
(4_2)_+(4_1)_+ 
&\implies (\bar{3}_2\bar{4}_1)_+=(\bar{4}_1\bar{3}_2)_+, \\
(4_1)_+(0_1)_+ 
&\implies (\bar{3}_1\bar{0}_1)_+=(\bar{0}_1\bar{3}_1)_+. &&
\end{align*}

The following is the exact condition for a partition of $C$ by circularly ordered subsets to be a vertex set that comes from a multiple planar diagram. The proposition can be proved by the same argument for Proposition 3 in \cite{lwwy1}, and is omitted here.

\begin{proposition}\label{avs2}
A disjoint union of circularly ordered subsets of the set $C$ of all corners \eqref{corners} is a vertex set if and only if the following are satisfied:
\begin{itemize}
\item Each (circularly ordered) subset has at least three corners.
\item $(i_p)_+((i-1)_p)_+\cdots$ is not a subset.
\item If $(i_p)_+(j_q)_+\cdots$ is a subset, then $((j+1)_q)_+((i-1)_p)_+\cdots$ is a subset. 
\item If $(i_p)_+(j_q)_-\cdots$ is a subset, then $((j-1)_q)_-((i-1)_p)_+\cdots$ is a subset.
\end{itemize}
\end{proposition}

Correspondingly, the following are the conditions for edge pairs in a planar diagram
\begin{itemize}
\item No degree one vertex: $(\bar{i}_p\overline{i+1}_p)_+$ is not an edge pair.
\item No degree two vertex in opposing way: If $(\bar{i}_p\bar{j}_q)_+$ is an edge pair, then $(\overline{i+1}_p\overline{j-1}_q)_+$ is not an edge pair.
\item No degree two vertex in twisted way: If $(\bar{i}_p\bar{j}_q)_-$ is an edge pair, then $(\overline{i+1}_p\overline{j+1}_q)_-$ is not an edge pair.
\end{itemize}

Two multiple planar diagrams represent the same tiling if and only if they are related by relabelling the prototile $n$-gon $P$, and relabelling the tiles. 
\begin{itemize}
\item Relabelling the prototile $P$: $c\in {\bb Z}_n$ is a constant
\begin{itemize}
\item Apply $i_p\mapsto (c+i)_p$ to all corners. This is the same as $\bar{i}_p\mapsto \overline{c+i}_p$ for all edges.
\item Apply $i_p\mapsto (c-i)_p$ to all corners. This is the same as $\bar{i}_p\mapsto \overline{c-i-1}_p$ for all edges.
\end{itemize}
\item Relabelling the tiles: For an invertible map $\phi$ of $\{1,2,\dots,r\}$ to itself, apply $i_p\mapsto i_{\phi(p)}$ and $\bar{i}_p\mapsto \bar{i}_{\phi(p)}$.
\end{itemize}
The decoration $\sigma$ is preserved in all cases.

\section{Tiling of Orientable Surface}
\label{orientable}

If the surface is orientable, then the algorithm for multiple planar diagrams and vertex sets can be simplified. We may fix an orientation of the surface, and compare the orientations of the individual tiles with the orientation of the surface. All tiles ${\mc T}$ are divided into two disjoint groups ${\mc T}_+$ and ${\mc T}_-$, where ${\mc T}_+$ consists of tiles of the same orientation as the surface, and ${\mc T}_-$ consists of tiles of different orientation from the surface. The disjoint union ${\mc T}={\mc T}_+\cup{\mc T}_-$ can be used to replace the subscript decorations $\sigma=\pm$. In fact, after relabelling the tiles, we may further assume 
\[
{\mc T}_+=\{T_1,T_2,\dots,T_s\},\quad
{\mc T}_-=\{T_{s+1},T_{s+2},\dots,T_f\}.
\]
Here we allow ${\mc T}_+$ to be empty ($s=0$) or ${\mc T}_-$ to be empty ($s=f$).

Since the surface is orientable, a pair of edges from tiles in the same ${\mc T}_{\pm}$ are opposing. Moreover, a pair of one edge in ${\mc T}_+$ and another edge in ${\mc T}_-$ are twisted. This means the following decorations of the edge pairs
\[
\bar{i}_p\bar{j}_q
=\begin{cases}
(\bar{i}_p\bar{j}_q)_+, &\text{if }1\le p,q\le s \text{ or } s+1\le p,q\le f \\
(\bar{i}_p\bar{j}_q)_-, &\text{if }1\le p\le s\text{ and }s+1\le q\le f
\end{cases}.
\]
Since the edge pair is not ordered, the second line also includes the case $s+1\le p\le f$ and $1\le q\le s$.

In further splitting the edge pair, we may always draw the pictures in Figure \ref{dtov} to reflect the two groups ${\mc T}_+$ and ${\mc T}_-$. In other words, $T_p$ for $1\le p\le s$ are always counterclockwise, and $T_p$ for $s+1\le p\le f$ are always clockwise. This means $i_p$ and $\bar{i}_p$ are always decorated by $\sigma=+$ for $1\le p\le s$ and by $\sigma=-$ for $s+1\le p\le f$. With this understanding, we may omit the decorations by keeping track of the two ranges of $p$. Then we may summarise the formulae \eqref{dtov_eq1+}, \eqref{dtov_eq1-}, \eqref{dtov_eq2} as the following:
\begin{itemize}
\item If $1\le p\le s$, then we find $\overline{i-1}_p\bar{j}_q\in D$, and get
\[
i_p\cdots=\begin{cases}
i_pj_q\cdots, &\text{if }1\le q\le s \\
i_p(j+1)_q\cdots, &\text{if }s+1\le q\le f
\end{cases}.
\]
\item If $s+1\le p\le f$, then we find $\overline{i}_p\bar{j}_q\in D$, and get
\[
i_p\cdots=\begin{cases}
i_pj_q\cdots, &\text{if }1\le q\le s \\
i_p(j+1)_q\cdots, &\text{if }s+1\le q\le f
\end{cases}.
\]
\end{itemize}
Conversely, the formula from vertices to edge pairs is
\[
i_pj_q\cdots \implies \begin{cases}
\overline{i-1}_p\bar{j}_q, 
&\text{if }1\le p,q\le s \\
\overline{i-1}_p\overline{j-1}_q, 
&\text{if }1\le p\le s\text{ and }s+1\le q\le f \\
\bar{i}_p\bar{j}_q, 
&\text{if }s+1\le p\le f\text{ and }1\le q\le s \\
\bar{i}_p\overline{j-1}_q, 
&\text{if }s+1\le p,q\le f
\end{cases}.
\]

Two multiple planar diagrams represent the same tiling if and only if they are related by relabelling the prototile $n$-gon $P$, and relabelling the tiles. Relabelling the prototile means $i\mapsto c+i$ or $i\mapsto c-i$, same as the general case. Relabelling the tiles means the preservation of the disjoint union ${\mc T}={\mc T}_+\cup {\mc T}_-$. Specifically, this is given by an invertible map $\phi$ of $\{1,2,\dots,f\}$ to itself, such that either ${\mc T}_+$ and ${\mc T}_-$ are preserved
\begin{align*}
\phi(\{1,2,\dots,s\})
&=\{1,2,\dots,s\}, \\
\phi(\{s+1,s+2,\dots,f\})
&=\{s+1,s+2,\dots,f\},
\end{align*}
or ${\mc T}_+$ and ${\mc T}_-$ are exchanged
\begin{align*}
\phi(\{1,2,\dots,s\})
&=\{f-s+1,f-s+2,\dots,f\}, \\
\phi(\{s+1,s+2,\dots,f\})
&=\{1,2,\dots,f-s\}.
\end{align*}
Due the the exchange, we may always assume $2s\ge f$. Then exchange happens only for the case $2s=f$.

Consider orientable the double planar diagram 
\begin{equation}\label{eq2}
D=(
\xar{0}_1\xar{2}_1,
\xar{4}_1\xar{6}_1,
\xar{0}_2\xar{2}_2,
\xar{4}_2\xar{6}_2,
\xar{1}_1\xar{1}_2,
\xar{3}_1\xar{3}_2,
\xar{5}_1\xar{5}_2).
\end{equation}
We need to distinguish two cases: $s=2$ and $s=1$. The left of Figure \ref{eg2} are the pictures for the double planar diagram. We have $s=2$ (two tiles having same orientation) in the upper left, and $s=1$ (two tiles having different orientation) in the lower left.

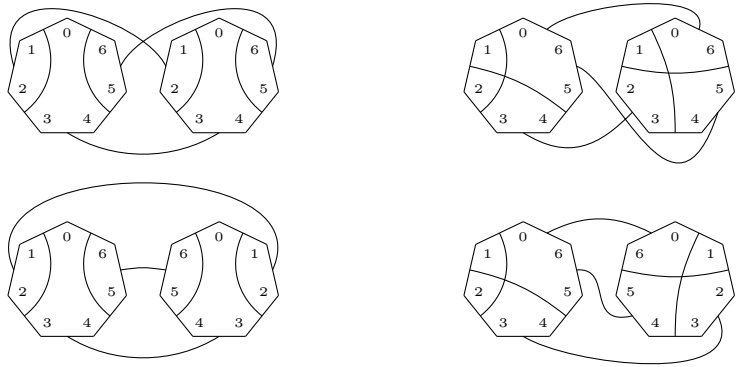
\begin{figure}[htp]
\centering
\begin{tikzpicture}[>=latex,scale=1]

\foreach \a in {0,...,6}
\foreach \u in {1,-1}
\foreach \x in {0,1}
\foreach \y in {0,1}
{
\begin{scope}[shift={(\u cm+6*\x cm, -2.7*\y cm)}]

\draw
	(38.57+51.43*\a:0.8) -- (90+51.43*\a:0.8);

\end{scope}
}

\foreach \a in {0,...,6}
\foreach \x in {0,1}
{

\foreach \u in {1,-1}
\node[shift={(\u cm+6*\x cm, 0 cm)}] at (90+51.43*\a:0.6) {\tiny \a};

\node[shift={(-1 cm+6*\x cm, -2.7 cm)}] at (90+51.43*\a:0.6) {\tiny \a};

\node[shift={(1 cm+6*\x cm, -2.7 cm)}] at (90-51.43*\a:0.6) {\tiny \a};

}

\foreach \a in {0,...,6}
\foreach \x in {0,1}
\foreach \y in {0,1}
{
\coordinate (A\a\x\y) at ([shift={(-1cm+6*\x cm, -2.7*\y cm)}] 115.71+51.43*\a:0.72);
\coordinate (B\a\x\y) at ([shift={(1cm+6*\x cm, -2.7*\y cm)}] 115.71+51.43*\a:0.72);
}


\foreach \a/\b in {0/2, 4/6}
\foreach \y in {0,1}
\draw
	(A\a 0\y) to[out=-64.29+51.43*\a, in=-64.29+51.43*\b] (A\b 0\y)
	(B\a 0\y) to[out=-64.29+51.43*\a, in=-64.29+51.43*\b] (B\b 0\y);
	
\foreach \a/\b in {0/2, 1/4}
\foreach \y in {0,1}
\draw
	(A\a 1\y) to[out=-64.29+51.43*\a, in=-64.29+51.43*\b] (A\b 1\y);

\foreach \a/\b/\y in {0/3/0, 1/5/0, 6/3/1, 1/5/1}
\draw
	(B\a 1\y) to[out=-64.29+51.43*\a, in=-64.29+51.43*\b] (B\b 1\y);

\draw
	(A100) .. controls (-2.1,1.5) and (0,0.8) .. (B100)
	(B500) .. controls (2.1,1.5) and (0,0.8) .. (A500)
	(A300) .. controls (-0.4,-1.1) and (0.4,-1.1) .. (B300);

\draw[shift={(0 cm,-2.7 cm)}]
	(A101) .. controls (-2.4,1.7) and (2.4,1.7) .. (B501)
	(A501) .. controls (0,0.2) .. (B101)
	(A301) .. controls (-0.4,-1.1) and (0.4,-1.1) .. (B301);

\draw[shift={(6 cm,0 cm)}]
	(A610) .. controls (-0.3,1) and (1.5,1.2) .. (B610)
	(A310) .. controls (-0.4,-1.1) and (0,-0.9) .. (B210)
	(A510) .. controls (0.2,0) and (1,-2.2) .. (B410);

\draw[shift={(6 cm,-2.7 cm)}]	
	(A611) .. controls (-0.2,0.9) and (0.2,0.9) .. (B011)	
	(A511) .. controls (0.2,0.2) and (-0.2,-0.6) .. (B211)
	(A311) .. controls (-0.4,-1.1) and (2,-1.4) .. (B411);
			
\end{tikzpicture}
\caption{Two versions of double planar diagrams \eqref{eq2} and \eqref{eq3}.} 
\label{eg2}
\end{figure}

For $s=2$, we have ${\mc T}_+=\{T_1,T_2\}$ and ${\mc T}_-=\emptyset$. There is only one rule to follow
\[
i_p\cdots
\overset{\scriptsize \overline{i-1}_p\bar{j}_q}{=\joinrel=\joinrel=\joinrel=}
i_pj_q\cdots.
\]
By the rule, we get vertices for the double planar diagram \eqref{eq2}
\begin{align*}
0_1\cdots
&\overset{\scriptsize \xar{6}_1\xar{4}_1}{=\joinrel=} 
0_14_1\cdots
\overset{\scriptsize \xar{3}_1\xar{3}_2}{=\joinrel=}
0_14_13_2\cdots 
\overset{\scriptsize \xar{2}_2\xar{0}_2}{=\joinrel=}
0_14_13_20_2\cdots 
\overset{\scriptsize \xar{6}_2\xar{4}_2}{=\joinrel=}
0_14_13_20_24_2\cdots \\
&\overset{\scriptsize \xar{3}_2\xar{3}_1}{=\joinrel=}
0_14_13_20_24_23_1\cdots
\overset{\scriptsize \xar{2}_1\xar{0}_1}{=\joinrel=}
0_14_13_20_24_23_1, \\
1_1\cdots
&\overset{\scriptsize \xar{0}_1\xar{2}_1}{=\joinrel=} 
1_12_1\cdots
\overset{\scriptsize \xar{1}_1\xar{1}_2}{=\joinrel=}
1_12_11_2\cdots 
\overset{\scriptsize \xar{0}_2\xar{2}_2}{=\joinrel=}
1_12_11_22_2\cdots  
\overset{\scriptsize \xar{1}_2\xar{1}_1}{=\joinrel=}
1_12_11_22_2, \\
5_1\cdots
&\overset{\scriptsize \xar{4}_1\xar{6}_1}{=\joinrel=}
5_16_1\cdots 
\overset{\scriptsize \xar{5}_1\xar{5}_2}{=\joinrel=}
5_16_15_2\cdots 
\overset{\scriptsize \xar{4}_2\xar{6}_1}{=\joinrel=}
5_16_15_26_2\cdots 
\overset{\scriptsize \xar{5}_2\xar{5}_1}{=\joinrel=}
5_16_15_26_2.
\end{align*}
We find the tiling has three vertices
\[
V=(0_14_13_20_24_23_1,
1_12_11_22_2,
5_16_15_26_2).
\]
The surface has Euler number $3-7+2=-2$. Since the surface is orientable, we get $S_D=2{\bb T}^2$.

For $s=1$, we have ${\mc T}_+=\{T_1\}$ and ${\mc T}_-=\{T_2\}$. Following the rules
\begin{align*}
i_1\cdots
&\overset{\scriptsize \overline{i-1}_1\bar{j}_1}{=\joinrel=\joinrel=\joinrel=}i_1j_1\cdots, &
i_2\cdots
&\overset{\scriptsize \bar{i}_2\bar{j}_1}{=\joinrel=\joinrel=\joinrel=}i_2j_1\cdots, \\
i_1\cdots
&\overset{\scriptsize \overline{i-1}_1\bar{j}_2}{=\joinrel=\joinrel=\joinrel=}i_1(j+1)_2\cdots,&
i_2\cdots
&\overset{\scriptsize \bar{i}_2\bar{j}_2}{=\joinrel=\joinrel=\joinrel=}i_2(j+1)_2\cdots,
\end{align*}
we get vertices for the double planar diagram \eqref{eq2}
\begin{align*}
0_1\cdots
&\overset{\scriptsize \xar{6}_1\xar{4}_1}{=\joinrel=} 
0_14_1\cdots
\overset{\scriptsize \xar{3}_1\xar{3}_2}{=\joinrel=} 
0_14_14_2\cdots
\overset{\scriptsize \xar{4}_2\xar{6}_2}{=\joinrel=} 
0_14_14_20_2\cdots
\overset{\scriptsize \xar{0}_2\xar{2}_2}{=\joinrel=} 
0_14_14_20_23_2\cdots \\
&\overset{\scriptsize \xar{3}_2\xar{3}_1}{=\joinrel=} 
0_14_14_20_23_23_1\cdots
\overset{\scriptsize \xar{2}_1\xar{0}_1}{=\joinrel=} 
0_14_14_20_23_23_1, \\
1_1\cdots
&\overset{\scriptsize \xar{0}_1\xar{2}_1}{=\joinrel=} 
1_12_1\cdots
\overset{\scriptsize \xar{1}_1\xar{1}_2}{=\joinrel=} 
1_12_12_2\cdots
\overset{\scriptsize \xar{2}_2\xar{0}_2}{=\joinrel=} 
1_12_12_21_2\cdots
\overset{\scriptsize \xar{1}_2\xar{1}_1}{=\joinrel=} 
1_12_12_21_2, \\
5_1\cdots
&\overset{\scriptsize \xar{4}_1\xar{6}_1}{=\joinrel=} 
5_16_1\cdots
\overset{\scriptsize \xar{5}_1\xar{5}_2}{=\joinrel=} 
5_16_16_2\cdots
\overset{\scriptsize \xar{6}_2\xar{4}_2}{=\joinrel=} 
5_16_16_25_2\cdots
\overset{\scriptsize \xar{5}_2\xar{5}_1}{=\joinrel=} 
5_16_16_25_2.
\end{align*}
The tiling has three vertices
\[
V=(0_14_14_20_23_23_1,
1_12_12_21_2,
5_16_16_25_2)
\]
and the surface is again $2{\bb T}^2$.

We may also easily see the vertex sets by following the rules in Figure \ref{p2v}.

We see that, for the same edge pairing, two tiles having the same orientation or different orientations leads to different tilings. In fact, the underlying surface may not even be the same. For example, the right of Figure \ref{eg2} shows the following pairing in the same orientation or different orientations 
\begin{equation}\label{eq3}
D=(
\xar{0}_1\xar{2}_1,
\xar{1}_1\xar{2}_1,
\xar{0}_2\xar{3}_2,
\xar{1}_2\xar{5}_2,
\xar{3}_1\xar{2}_2,
\xar{5}_1\xar{4}_2,
\xar{6}_1\xar{6}_2).
\end{equation}
They have respective vertex sets
\[
(0_13_13_21_26_2,
1_12_14_15_22_25_2,
6_14_20_2),\quad
(0_13_12_25_26_16_21_23_24_12_11_15_14_20_2).
\]
The numbers of vertices are $3$ and $1$, and the surfaces are $2{\bb T}^2$ and $3{\bb T}^2$.

\section{Geometrical Realisation}

Using the codes in Sections \ref{tiling} and \ref{orientable}, we may create computer program to find all the combinatorial tilings of a fixed surface of Euler characteristic $<0$, by a fixed number $f$ of congruent $n$-gons with $n\ge 7$. The program may either start from edge pairings, or vertex sets satisfying Lemma \ref{avs2}. 

In \cite{lwwy1}, we show that it is likely that all combinatorial single tile tilings are geometrically relisable. In other words, the prototiles can be realised by hyperbolic polygons with straight (geodesic) edges. For multiple tile tilings, the geometrical condition is much more complicated. For example, the number of distinct edge lengths in a tiling by a single $n$-gon is $\frac{n}{2}$. However, both $s=2$ and $s=1$ versions of the double planar diagram \eqref{eq2} satisfy
\[
|\bar{0}|=|\bar{2}|,\quad
|\bar{4}|=|\bar{6}|,\quad
[0]+[4]+[3]=[1]+[2]=[5]+[6]=\pi.
\]
Therefore the number of distinct edge lengths is five. In fact, there is no tiling of $2{\bb T}^2$ by two congruent heptagons, such that all seven edges have distinct edge lengths. On the other hand, there are many double tile tilings of $2{\bb T}^2$, such that all edge lengths must be the same, i.e., the prototile is equilateral.

The combinatorial algorithm produces three tilings of $2{\bb T}^2$ by two congruent heptagons, in the second row of Figure \ref{T2A}. However, the double planar diagram in the box
\[
D=(\bar{0}_1\bar{3}_1,\bar{3}_2\bar{6}_2,\bar{1}_1\bar{1}_2,\bar{2}_1\bar{2}_2,\bar{4}_1\bar{4}_2,\bar{5}_1\bar{5}_2,\bar{6}_1\bar{0}_2),\quad s=2
\]
has vertices
\[
(0_10_23_22_11_26_15_24_1,
1_13_12_2,
4_26_25_1).
\]
Then we get angle sum equalities
\[
2[0]+[1]+[2]+[3]+[4]+[5]+[6]
=[1]+[2]+[3]
=[4]+[5]+[6]
=2\pi.
\]
This is clearly a contradiction. 

Therefore we add the condition that the angle sum equalities admit positive angle value solutions in our computer program. Using this program, we get all double tile tilings of $3{\bb P}^2$, $2{\bb T}^2$, $4{\bb P}^2$, and double tile tilings of $3{\bb T}^2$ for $7\le n\le 13$. Table \ref{tiling_number} shows the numbers of such tilings, according to the numbers of distinct edge lengths. For the orientable surfaces $2{\bb T}^2$ and $3{\bb T}^2$, the table further shows two rows of numbers: The first row counts the number of tilings in which two tiles have the same orientation (i.e., $s=2$), and second row counts the number of tilings in which two tiles have different orientation (i.e., $s=1$). We observe that the two rows of numbers are at the similar scale. Moreover, there are only two double tile tilings (by heptagons and octagons) in the table that allow all edge lengths to be distinct. In Section \ref{edge}, we have a detailed discussion about distinct edge lengths.

\begin{table}[htp]
	\centering
	\begin{tabular}{|c|c||c|c|c|c|c|c|c|c|c|c||c|}
	\hline
	\multirow{2}{*}{surface} & \multirow{2}{*}{gon} & \multicolumn{10}{|c||}{number of edge lengths} & \multirow{2}{*}{total} \\
	\cline{3-12}
	& & 10 & 9 & 8 & 7 & 6 & 5 & 4 & 3 & 2 & 1 & \\
\hline \hline
	\multirow{3}{*}{$\begin{matrix}3{\bb P}^2 \\ \chi=-1\end{matrix} $} 
	& $7$ & &&  &  & 1 & 6 & 18 & 85 & 191 & 142 & 443 \\
	\cline{2-13}
	& $8$ & &&  &  & 1 & 6 & 18 & 71 & 158 & 104 & 358 \\
	\cline{2-13}
	& $9$ & &&  &  &  &  &  & 16 & 16 & 16 & 48 \\
	\hline \hline
	\multirow{12}{*}{$\begin{matrix}2{\bb T}^2 \\ \chi=-2\end{matrix} $} & \multirow{2}{*}{7} 
	&  &  &  &  & 3 & 2 & 20 & 49 & 110 & 106 & 290\\
	\cline{3-13}
	&& &  &  &  &  & 8 & 20 & 98 & 115 & 104 & 345\\
	\cline{2-13}
	& \multirow{2}{*}{8} 
	&  &  &  & 2 & 6 & 32 & 26 & 105 & 215 & 208 & 594\\
	\cline{3-13}
	&&  &  &  &  & 6 & 13 & 75 & 155 & 248 & 129 & 626\\
	\cline{2-13}
	& \multirow{2}{*}{9} 
	&  &  &  & 6 & 3 & 27 & 22 & 113 & 302 & 271 & 744\\
	\cline{3-13}
	&&  &  &  &  & 6 & 29 & 54 & 266 & 333 & 239 & 927\\
	\cline{2-13}
	& \multirow{2}{*}{10} 
	&  &  & 3 & 6 & 36 & 10 & 49 & 116 & 272 & 220 & 712\\
	\cline{3-13}
	&&  &  &  &  & 11 & 31 & 97 & 212 & 299 & 139 & 789\\
	\cline{2-13}
	& \multirow{2}{*}{11} 
	&  &  & 4 &  & 11 &  & 9 & 41 & 108 & 97 & 270\\
	\cline{3-13}
	&&  &  &  &  & 6 & 13 & 39 & 137 & 114 & 75 & 384\\
	\cline{2-13}
	& \multirow{2}{*}{12} 
	&  & 3 &  & 16 &  & 4 & 10 & 19 & 41 & 32 & 125\\
	\cline{3-13}
	&&  &  & &  & 15 & 6 & 41 & 38 & 36 & 23 & 159\\
	\hline \hline
	\multirow{6}{*}{$\begin{matrix}4{\bb P}^2 \\ \chi=-2\end{matrix} $} 
	& $7$ & &&  & 1 & 22 & 152 & 725 & 3179 & 6947 & 5542 & 16568 \\
	\cline{2-13}
	& $8$ & && 1 & 12 & 99 & 456 & 1936 & 7269 & 15952 & 12433 & 38158 \\
	\cline{2-13}
	& $9$ & && 1 & 19 & 168 & 609 & 2699 & 10596 & 21192 & 15831 & 51115 \\
	\cline{2-13}
	& $10$ & && 2 & 65 & 309 & 703 & 2857 & 9457 & 16902 & 11910 & 42205 \\
	\cline{2-13}
	& $11$ & && 3 & 57 & 126 & 278 & 1385 & 4513 & 7918 & 5472 & 19752 \\
	\cline{2-13}
	& $12$ & && 22 & 82 & 112 & 174 & 589 & 1323 & 1992 & 1249 & 5543 \\
	\hline \hline
	\multirow{14}{*}{$\begin{matrix}3{\bb T}^2 \\ \chi=-4\end{matrix} $} & \multirow{2}{*}{7}  
	&  &  &  & 1 &  & 17 & 31 & 155 & 254 & 227 & 685\\
	\cline{3-13}
	&&  &  &  &  &  &  & 61 & 137 & 286 & 168 & 652\\
	\cline{2-13}
	& \multirow{2}{*}{8} 
	&  &  & 1 & 3 & 27 & 116 & 339 & 1609 & 2758 & 2351 & 7204\\
	\cline{3-13}
	&&  &  &  &  &  & 159 & 641 & 2144 & 2706 & 1503 & 7153\\
	\cline{2-13}
	& \multirow{2}{*}{9} 
	&  &  & 5 & 8 & 216 & 564 & 3354 & 15386 & 29625 & 22745 & 71903\\
	\cline{3-13}
	&&  &  &  &  & 187 & 669 & 6275 & 17629 & 31812 & 19526 & 76098\\
	\cline{2-13}
	& \multirow{2}{*}{10} 
	&  & 3 & 24 & 219 & 602 & 2572 & 13029 & 55235 & 110039 & 84947 & 266670
\\
	\cline{3-13}
	&&  &  &  & 164 & 540 & 4941 & 25146 & 79656 & 106694 & 56447 & 273588\\
	\cline{2-13}
	& \multirow{2}{*}{11} 
	&  & 25 & 42 & 909 & 1166 & 8117 & 49336 & 197870 & 381457 & 277743 & 916665\\
	\cline{3-13}
	&&  &  & 66 & 323 & 2539 & 12469 & 89248 & 249183 & 403063 & 240664 & 997555\\
	\cline{2-13}
	& \multirow{2}{*}{12} 
	& 11 & 81 & 702 & 1171 & 4084 & 18091 & 95665 & 353939 & 677906 & 491760 & 1643410\\
	\cline{3-13}
	&&  & 26 & 141 & 1157 & 4313 & 35527 & 168947 & 490930 & 648491 & 339986 & 1689518\\
 \cline{2-13}
	& \multirow{2}{*}{13} 
	& 65 & 76 & 1656 & 916 & 5521 & 29773 & 152878 & 575257 & 1027213 & 709101 & 2502456\\
	\cline{3-13}
	&&  & 34 & 349 & 1264 & 9504 & 46186 & 277710 & 711624 & 1074943 & 604979 & 2726593 \\ 
	\cline{2-13}
	& $\vdots$ &  \multicolumn{11}{|c|}{$\vdots$} \\
	\hline
	\end{tabular}
\caption{Number of double tile tilings of $3{\bb P}^2$, $2{\bb T}^2$, $4{\bb P}^2$, $3{\bb T}^2$. For orientable $2{\bb T}^2$, $3{\bb T}^2$,  and each $n$-gon, the two rows count the tilings in which two tiles have the same or different orientations. For $3{\bb T}^2$, the data for $14\le n\le 18$ are not presented.}
\label{tiling_number}
\end{table}

Next, we draw the schematics of some tilings:
\begin{itemize}
\item Figure \ref{T2A}: all tilings of $2{\bb T}^2$ by two congruent heptagons, such that the number of distinct edge lengths is $\ge 5$. 
\item Figure \ref{P2B}: all tilings of $3{\bb P}^2$ by two heptagons or octagons, with 5 distinct edge lengths. 
\item Figure \ref{P2A}: all tilings of $3{\bb P}^2$ and $4{\bb P}^2$ with maximal number of distinct edge lengths, except $4{\bb P}^2$ by two $12$-gons (there are too many to draw). 
\item Figure \ref{P2C}: all tilings of $3{\bb P}^2$ by two $9$-gons with all edge length being equal. 
\end{itemize}
In the pictures, the thick edges of the same color have the same lengths, and the normal edges can have distinct lengths. We also indicate the vertices by colors.

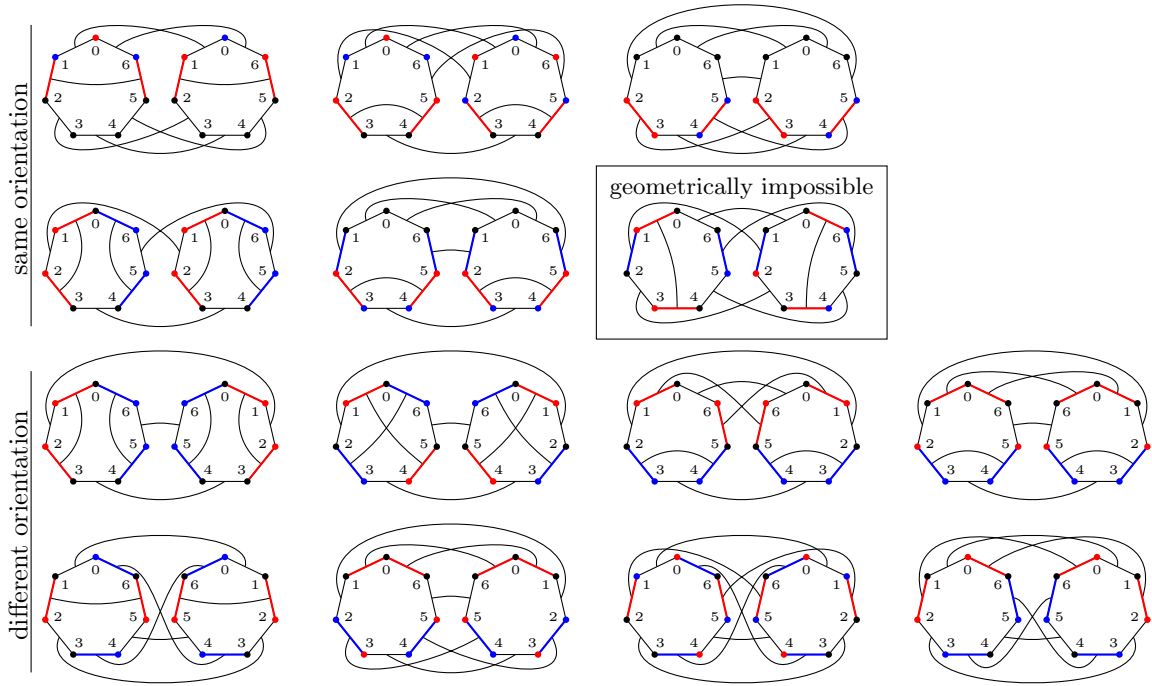
\begin{figure}[htp]
\centering
\begin{tikzpicture}[>=latex,scale=0.85]

\node[rotate=90] at (-2.2,-1.4) {same orientation};
\node[rotate=90] at (-2.2,-6.8) {different orientation};
\node at (9,-1.55) {\small geometrically impossible};
\draw
	(-2,1) -- ++(0,-4.7)
	(-2,-4.4) -- ++(0,-4.7)
	(6.75,-3.9) rectangle (11.25,-1.2);


\foreach \a in {0,...,6}
\foreach \u in {1,-1}
\foreach \x in {0,1,2}
\foreach \y in {0,1}
{
\begin{scope}[shift={(\u cm+4.5*\x cm, -2.7*\y cm)}]

\draw
	(38.57+51.43*\a:0.8) -- (90+51.43*\a:0.8);
	
\node at (90+51.43*\a:0.6) {\tiny \a};

\end{scope}
}

\foreach \a in {0,...,6}
\foreach \x in {0,1,2}
\foreach \y in {0,1}
{
\coordinate (A\a\x\y) at ([shift={(-1cm+4.5*\x cm, -2.7*\y cm)}] 115.71+51.43*\a:0.72);
\coordinate (B\a\x\y) at ([shift={(1cm+4.5*\x cm, -2.7*\y cm)}] 115.71+51.43*\a:0.72);
}


\foreach \x/\y in {0/0, 1/0, 2/0, 0/1, 1/1}
\draw[shift={(4.5*\x cm, -2.7*\y cm)}]
	(A3\x\y) .. controls (-0.4,-1.1) and (0.4,-1.1) .. (B3\x\y);

\foreach \x/\y in {1/0, 2/0, 0/0, 1/1}
\draw[shift={(4.5*\x cm, -2.7*\y cm)}]
	(A0\x\y) .. controls (-1.5,1.2) and (0.3,1) .. (B0\x\y)
	(A6\x\y) .. controls (-0.3,1) and (1.5,1.2) .. (B6\x\y);

\foreach \x/\y in {0/0, 2/0, 2/1}
\draw[shift={(4.5*\x cm, -2.7*\y cm)}]
	(A2\x\y) .. controls (-2,-1.3) and (-0.3,-0.9) .. (B2\x\y)
	(A4\x\y) .. controls (0.3,-0.9) and (2,-1.3) .. (B4\x\y);

\foreach \x/\y in {0/1, 1/0, 2/1}
\draw[shift={(4.5*\x cm, -2.7*\y cm)}]
	(A1\x\y) .. controls (-2.1,1.5) and (0,0.8) .. (B1\x\y)
	(B5\x\y) .. controls (2.1,1.5) and (0,0.8) .. (A5\x\y);

\foreach \x/\y in {2/0, 1/1}	
\draw[shift={(4.5*\x cm, -2.7*\y cm)}]	
	(A1\x\y) .. controls (-2.4,1.7) and (2.4,1.7) .. (B5\x\y)
	(A5\x\y) .. controls (0,0.2) .. (B1\x\y);

\foreach \a/\b/\x/\y in {0/2/0/1, 4/6/0/1, 2/4/1/0, 2/4/1/1, 1/5/0/0}
\draw
	(A\a\x\y) to[out=-64.29+51.43*\a, in=-64.29+51.43*\b] (A\b\x\y)
	(B\a\x\y) to[out=-64.29+51.43*\a, in=-64.29+51.43*\b] (B\b\x\y);

\draw[shift={(9 cm, -2.7 cm)}]
	(A021) to[out=-64.29+51.43*0, in=-64.29+51.43*3] (A321)
	(B621) to[out=-64.29+51.43*6, in=-64.29+51.43*3] (B321)
	(A621) .. controls (-0.2,0.9) and (0.2,0.9) .. (B021);


\foreach \a/\x/\y in {1/0/0, 5/0/0, 2/1/0, 4/1/0, 2/2/0, 4/2/0, 0/0/1, 2/0/1, 2/1/1, 4/1/1, 3/2/1}
\foreach \u in {-1,1}
\draw[red, thick, shift={(\u cm + 4.5*\x cm, -2.7*\y cm)}]
	(90+51.43*\a:0.8) -- (90+51.43+51.43*\a:0.8);
	
\foreach \a/\x/\y in {4/0/1, 6/0/1, 1/1/1, 5/1/1}
\foreach \u in {-1,1}
\draw[blue, thick, shift={(\u cm + 4.5*\x cm, -2.7*\y cm)}]
	(90+51.43*\a:0.8) -- (90+51.43+51.43*\a:0.8);

\draw[red, thick, shift={(8 cm, -2.7 cm)}]
	(90+51.43*0:0.8) -- (90+51.43+51.43*0:0.8);
\draw[red, thick, shift={(10 cm, -2.7 cm)}]
	(90+51.43*6:0.8) -- (90+51.43+51.43*6:0.8);
\draw[blue, thick, shift={(8 cm, -2.7 cm)}]
	(90+51.43*1:0.8) -- (90+51.43+51.43*1:0.8)
	(90+51.43*5:0.8) -- (90+51.43+51.43*5:0.8);
\draw[blue, thick, shift={(10 cm, -2.7 cm)}]
	(90+51.43*1:0.8) -- (90+51.43+51.43*1:0.8)
	(90+51.43*5:0.8) -- (90+51.43+51.43*5:0.8);

	
\foreach \a/\u/\x/\y in {
	0/-1/0/0, 1/1/0/0, 6/1/0/0,
	0/-1/1/0, 2/-1/1/0, 5/-1/1/0, 1/1/1/0, 6/1/1/0,
	2/-1/2/0, 3/-1/2/0, 2/1/2/0, 3/1/2/0,
	1/-1/0/1, 2/-1/0/1, 1/1/0/1, 2/1/0/1,
	2/-1/1/1, 5/-1/1/1, 2/1/1/1, 5/1/1/1,
	1/-1/2/1, 3/-1/2/1, 2/1/2/1}
\fill[red, shift={(\u cm + 4.5*\x cm, -2.7*\y cm)}]
	(90+51.43*\a:0.8) circle (0.05);

\foreach \a/\u/\x/\y in {
	0/1/0/0, 1/-1/0/0, 6/-1/0/0,
	0/1/1/0, 2/1/1/0, 5/1/1/0, 1/-1/1/0, 6/-1/1/0,
	4/-1/2/0, 5/-1/2/0, 4/1/2/0, 5/1/2/0,
	5/-1/0/1, 6/-1/0/1, 5/1/0/1, 6/1/0/1,
	3/-1/1/1, 4/-1/1/1, 3/1/1/1, 4/1/1/1,
	6/1/2/1, 4/1/2/1, 5/-1/2/1}
\fill[blue, shift={(\u cm + 4.5*\x cm, -2.7*\y cm)}]
	(90+51.43*\a:0.8) circle (0.05);

\foreach \a/\u/\x/\y in {
	2/-1/0/0, 3/-1/0/0, 4/-1/0/0, 5/-1/0/0, 2/1/0/0, 3/1/0/0, 4/1/0/0, 5/1/0/0,
	3/-1/1/0, 4/-1/1/0, 3/1/1/0, 4/1/1/0,
	0/-1/2/0, 1/-1/2/0, 6/-1/2/0, 0/1/2/0, 1/1/2/0, 6/1/2/0,
	0/-1/0/1, 3/-1/0/1, 4/-1/0/1, 0/1/0/1, 3/1/0/1, 4/1/0/1,
	0/-1/1/1, 1/-1/1/1, 6/-1/1/1, 0/1/1/1, 1/1/1/1, 6/1/1/1,
	0/-1/2/1, 2/-1/2/1, 4/-1/2/1, 6/-1/2/1, 0/1/2/1, 1/1/2/1, 3/1/2/1, 5/1/2/1}
\fill[shift={(\u cm + 4.5*\x cm, -2.7*\y cm)}]
	(90+51.43*\a:0.8) circle (0.05);
										

\begin{scope}[yshift=-5.4cm]

\foreach \a in {0,...,6}
\foreach \u in {1,-1}
\foreach \x in {0,...,3}
\foreach \y in {0,1}
{
\begin{scope}[shift={(\u cm+4.5*\x cm, -2.7*\y cm)}]

\draw
	(38.57+51.43*\a:0.8) -- (90+51.43*\a:0.8);
	
\node at (90-51.43*\a*\u:0.6) {\tiny \a};

\end{scope}
}

\foreach \a in {0,...,6}
\foreach \x in {0,...,3}
\foreach \y in {0,1,2}
{

\coordinate (A\a\x\y) at ([shift={(-1cm+4.5*\x cm, -2.7*\y cm)}] 115.71+51.43*\a:0.72);
\coordinate (B\a\x\y) at ([shift={(1cm+4.5*\x cm, -2.7*\y cm)}] 64.28-51.43*\a:0.72);

}

\foreach \a/\b/\x/\y in 
	{0/2/0/0, 4/6/0/0, 0/4/1/0, 2/6/1/0, 2/4/2/0, 2/4/3/0, 1/5/0/1}
\draw
	(A\a\x\y) to[out=295.71+51.43*\a, in=295.71+51.43*\b] (A\b\x\y)
	(B\a\x\y) to[out=244.28-51.43*\a, in=244.28-51.43*\b] (B\b\x\y);

\foreach \x/\y in {0/1, 2/1, 3/1}
\draw[shift={(4.5*\x cm, -2.7*\y cm)}]
	(A0\x\y) .. controls (-1.5,1.3) and (1.5,1.3) .. (B0\x\y);
	
\foreach \x/\y in {0/0, 1/0, 2/0, 3/0, 1/1}
\draw[shift={(4.5*\x cm, -2.7*\y cm)}]
	(A1\x\y) .. controls (-2.4,1.7) and (2.4,1.7) .. (B1\x\y);

\foreach \x/\y in {0/1, 2/1, 3/1}
\draw[shift={(4.5*\x cm, -2.7*\y cm)}]
	(A2\x\y) .. controls (-2,-1.4) and (2,-1.4) .. (B2\x\y);
	
\foreach \x/\y in {0/0, 1/0, 2/0, 3/0, 1/1}
\draw[shift={(4.5*\x cm, -2.7*\y cm)}]
	(A3\x\y) .. controls (-0.4,-1.1) and (0.4,-1.1) .. (B3\x\y);

\foreach \x/\y in {0/1, 2/1, 3/1}
\draw[shift={(4.5*\x cm, -2.7*\y cm)}]
	(A4\x\y) .. controls (0,-0.5) .. (B4\x\y);
	
\foreach \x/\y in {0/0, 1/0, 3/0, 1/1}
\draw[shift={(4.5*\x cm, -2.7*\y cm)}]
	(A5\x\y) .. controls (0,0.2) .. (B5\x\y);
		
\foreach \x/\y in {3/0, 1/1}
\draw[shift={(4.5*\x cm, -2.7*\y cm)}]
	(A0\x\y) .. controls (-1.5,1.2) and (0.3,1) .. (B6\x\y)
	(B0\x\y) .. controls (1.5,1.2) and (-0.3,1) .. (A6\x\y);

\foreach \x/\y in {0/1, 2/1}
\draw[shift={(4.5*\x cm, -2.7*\y cm)}]
	(A3\x\y) .. controls (0,-1.5) and (-0.1,0.9) .. (B6\x\y)
	(B3\x\y) .. controls (0,-1.5) and (0.1,0.9) .. (A6\x\y);
	
\draw[shift={(4.5*2 cm, -2.7*0 cm)}]
	(A020) .. controls (-0.8,1.5) and (0,0.4) .. (B520)
	(B020) .. controls (0.8,1.5) and (0,0.4) .. (A520)
	(A620) .. controls (-0.2,0.9) and (0.2,0.9) .. (B620);
			
\draw[shift={(4.5*1 cm, -2.7*1 cm)}]
	(A211) .. controls (-2,-1.3) and (-0.3,-0.9) .. (B411)
	(A411) .. controls (0.3,-0.9) and (2,-1.3) .. (B211);

\draw[shift={(4.5*2 cm, -2.7*1 cm)}]
	(A121) .. controls (-2.1,1.5) and (0,0.8) .. (B521)
	(B121) .. controls (2.1,1.5) and (0,0.8) .. (A521);

\draw[shift={(4.5*3 cm, -2.7*1 cm)}]
	(A131) .. controls (-2.1,1.5) and (0,1.2) .. (B631)
	(B131) .. controls (2.1,1.5) and (0,1.2) .. (A631)
	(A331) .. controls (-0.5,-1.5) and (0,0.2) .. (B531)
	(B331) .. controls (0.5,-1.5) and (0,0.2) .. (A531);


\foreach \a/\x/\y in {0/0/0, 2/0/0, 0/1/0, 4/1/0, 0/2/0, 5/2/0, 0/3/0, 6/3/0, 1/0/1, 5/0/1, 0/1/1, 6/1/1, 1/2/1, 5/2/1, 1/3/1, 6/3/1}
\foreach \u in {1,-1}
\draw[red, thick, shift={(\u cm+4.5*\x cm, -2.7*\y cm)}]
	(90-51.43*\u*\a:0.8) -- (90-51.43*\u*\a-51.43*\u:0.8);

\foreach \a/\x/\y in {4/0/0, 6/0/0, 2/1/0, 6/1/0, 2/2/0, 4/2/0, 2/3/0, 4/3/0, 3/0/1, 6/0/1, 2/1/1, 4/1/1, 3/2/1, 6/2/1, 3/3/1, 5/3/1}
\foreach \u in {1,-1}
\draw[blue, thick, shift={(\u cm+4.5*\x cm, -2.7*\y cm)}]
	(90-51.43*\u*\a:0.8) -- (90-51.43*\u*\a-51.43*\u:0.8);


\foreach \u in {1,-1}
{

\foreach \a/\x/\y in {
	1/0/0, 2/0/0,
	1/1/0, 4/1/0,
	1/2/0, 6/2/0,
	2/3/0, 5/3/0,
	2/0/1, 5/0/1,
	3/1/1, 5/1/1,
	0/2/1, 4/2/1,
	0/3/1, 2/3/1
	}
\fill[red, shift={(\u cm+4.5*\x cm, -2.7*\y cm)}]
	(90-51.43*\u*\a:0.8) circle (0.05);

\foreach \a/\x/\y in {
	5/0/0, 6/0/0,
	3/1/0, 6/1/0,
	3/2/0, 4/2/0,
	3/3/0, 4/3/0,
	0/0/1, 4/0/1,
	2/1/1, 4/1/1,
	1/2/1, 5/2/1,
	3/3/1, 5/3/1
	}
\fill[blue, shift={(\u cm+4.5*\x cm, -2.7*\y cm)}]
	(90-51.43*\u*\a:0.8) circle (0.05);
		
\foreach \a/\x/\y in {
	0/0/0, 3/0/0, 4/0/0, 
	0/1/0, 2/1/0, 5/1/0, 
	0/2/0, 2/2/0, 5/2/0, 
	0/3/0, 1/3/0, 6/3/0,
	1/0/1, 3/0/1, 6/0/1,
	1/1/1, 0/1/1, 6/1/1,
	2/2/1, 3/2/1, 6/2/1,
	1/3/1, 4/3/1, 6/3/1
	}
\fill[shift={(\u cm+4.5*\x cm, -2.7*\y cm)}]
	(90-51.43*\u*\a:0.8) circle (0.05);
	
}

\end{scope}
			
\end{tikzpicture}
\caption{Tilings of $2{\bb T}^2$ by two congruent $7$-gons, with 5 or 6 edge lengths. Thick edges of the same color have the same lengths. Normal edges have distinct lengths.} 
\label{T2A}
\end{figure}

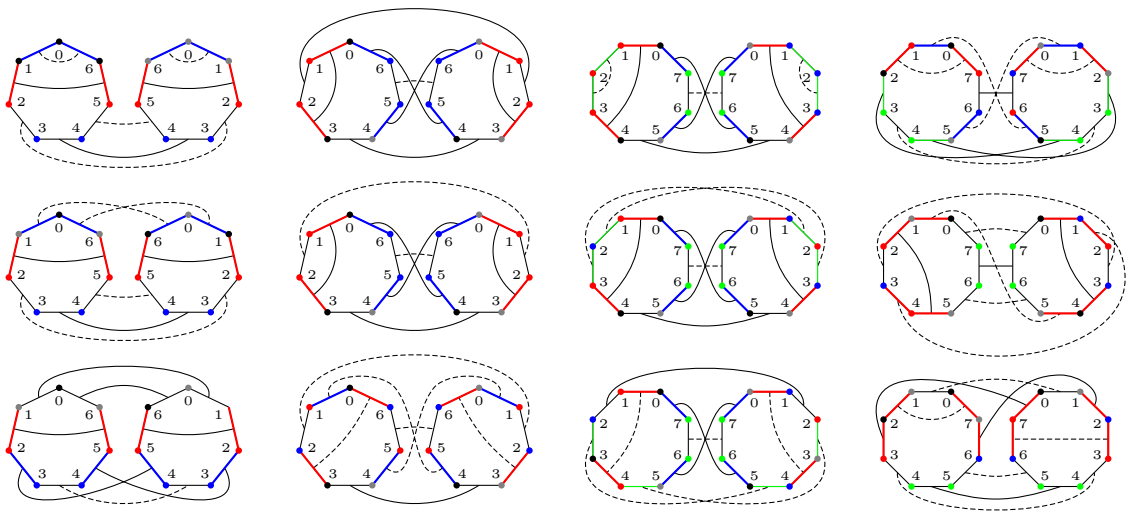
\begin{figure}[htp]
\centering
\begin{tikzpicture}[>=latex,scale=0.85]

\foreach \a in {0,...,6}
\foreach \u in {1,-1}
\foreach \x in {0,1}
\foreach \y in {0,1,2}
{
\begin{scope}[shift={(\u cm+4.5*\x cm, -2.7*\y cm)}]

\draw
	(38.57+51.43*\a:0.8) -- (90+51.43*\a:0.8);
	
\node at (90-51.43*\a*\u:0.6) {\tiny \a};

\end{scope}
}

\foreach \a in {0,...,6}
\foreach \x in {0,1,2}
\foreach \y in {0,1,2}
{

\coordinate (A\a\x\y) at ([shift={(-1cm+4.5*\x cm, -2.7*\y cm)}] 115.71+51.43*\a:0.72);
\coordinate (B\a\x\y) at ([shift={(1cm+4.5*\x cm, -2.7*\y cm)}] 64.28-51.43*\a:0.72);

}



\foreach \a/\b/\x/\y in 
	{1/5/0/0, 1/5/0/1, 1/5/0/2,
	0/2/1/0, 0/2/1/1}
\draw
	(A\a\x\y) to[out=295.71+51.43*\a, in=295.71+51.43*\b] (A\b\x\y)
	(B\a\x\y) to[out=244.28-51.43*\a, in=244.28-51.43*\b] (B\b\x\y);
	

\foreach \x/\y in {1/0}
\draw[shift={(4.5*\x cm, -2.7*\y cm)}]
	(A1\x\y) .. controls (-2.4,1.7) and (2.4,1.7) .. (B1\x\y);
	
\foreach \x/\y in {0/0, 0/1, 1/0, 1/1, 1/2}
\draw[shift={(4.5*\x cm, -2.7*\y cm)}]
	(A3\x\y) .. controls (-0.4,-1.1) and (0.4,-1.1) .. (B3\x\y);
	
\draw[shift={(4.5*0 cm, -2.7*2 cm)}]
	(A002) .. controls (-1.5,1.3) and (1.5,1.3) .. (B002)
	(A602) .. controls (-0.2,0.9) and (0.2,0.9) .. (B602)
	(A202) .. controls (-2,-1.3) and (-0.3,-0.9) .. (B402)
	(A402) .. controls (0.3,-0.9) and (2,-1.3) .. (B202);

\foreach \x/\y in {1/0, 1/1}
\draw[shift={(4.5*\x cm, -2.7*\y cm)}]
	(A6\x\y) .. controls (-0.1,0.9) and (0.1,-0.7) .. (B4\x\y)
	(A4\x\y) .. controls (-0.1,-0.7) and (0.1,0.9) .. (B6\x\y);
		

\begin{scope}[dash pattern=on 2pt off 1pt]


\foreach \a/\b/\x/\y in 
	{0/6/0/0, 
	2/6/1/2}
\draw
	(A\a\x\y) to[out=295.71+51.43*\a, in=295.71+51.43*\b] (A\b\x\y)
	(B\a\x\y) to[out=244.28-51.43*\a, in=244.28-51.43*\b] (B\b\x\y);


\foreach \x/\y in {1/1, 1/2}
\draw[shift={(4.5*\x cm, -2.7*\y cm)}]
	(A1\x\y) .. controls (-2.4,1.7) and (2.4,1.7) .. (B1\x\y);
	
\draw[shift={(4.5*0 cm, -2.7*2 cm)}]
	(A302) .. controls (-0.4,-1.1) and (0.4,-1.1) .. (B302);

\foreach \x/\y in {0/0, 0/1}
\draw[shift={(4.5*\x cm, -2.7*\y cm)}]
	(A4\x\y) .. controls (0,-0.5) .. (B4\x\y)
	(A2\x\y) .. controls (-2,-1.4) and (2,-1.4) .. (B2\x\y);

\foreach \x/\y in {1/0, 1/1, 1/2}
\draw[shift={(4.5*\x cm, -2.7*\y cm)}]
	(A5\x\y) .. controls (0,0.2) .. (B5\x\y);

\draw[shift={(4.5*0 cm, -2.7*1 cm)}]
	(A001) .. controls (-1.5,1.2) and (0.3,1) .. (B601)
	(A601) .. controls (-0.3,1) and (1.5,1.2) .. (B001);

\draw[shift={(4.5*1 cm, -2.7*2 cm)}]
	(A012) to[out=115.71, in=130] 
	(-0.3,0.8) to[out=-50, in=-141.44]
	(B412)
	(A412) to[out=321.43, in=230] 
	(0.3,0.8) to[out=50, in=64.28]
	(B012);
			
\end{scope}


\foreach \a/\b/\x/\y in {1/5/0/0, 1/5/0/1, 1/5/0/2, 0/2/1/0, 0/2/1/1, 2/6/1/2}
\foreach \u in {1,-1}
\draw[red, thick]
	([shift={(\u cm + 4.5*\x cm, -2.7*\y cm)}] 90-51.43*\u*\a:0.8) -- ([shift={(\u cm + 4.5*\x cm, -2.7*\y cm)}] 90-51.43*\u*\a-51.43*\u:0.8)
	([shift={(\u cm + 4.5*\x cm, -2.7*\y cm)}] 90-51.43*\u*\b:0.8) -- ([shift={(\u cm + 4.5*\x cm, -2.7*\y cm)}] 90-51.43*\u*\b-51.43*\u:0.8);
	
\foreach \a/\b/\x/\y in {0/6/0/0, 0/6/0/1, 2/4/0/2, 4/6/1/0, 4/6/1/1, 0/4/1/2}
\foreach \u in {1,-1}
\draw[blue, thick]
	([shift={(\u cm + 4.5*\x cm, -2.7*\y cm)}] 90-51.43*\u*\a:0.8) -- ([shift={(\u cm + 4.5*\x cm, -2.7*\y cm)}] 90-51.43*\u*\a-51.43*\u:0.8)
	([shift={(\u cm + 4.5*\x cm, -2.7*\y cm)}] 90-51.43*\u*\b:0.8) -- ([shift={(\u cm + 4.5*\x cm, -2.7*\y cm)}] 90-51.43*\u*\b-51.43*\u:0.8);
		
	
\foreach \a/\u/\x/\y in {
	0/-1/0/0, 1/-1/0/0, 6/-1/0/0,
	0/-1/0/1, 1/1/0/1, 6/1/0/1,
	0/-1/0/2, 1/1/20/2, 6/1/0/2,
	0/-1/1/0, 3/-1/1/0, 4/1/1/0,
	0/-1/1/1, 3/-1/1/1, 4/1/1/1,
	0/-1/1/2, 3/-1/1/2, 4/1/1/2
	}
\fill
	([shift={(\u cm + 4.5*\x cm, -2.7*\y cm)}] 90-51.43*\u*\a:0.8) circle (0.05);

\foreach \a/\u/\x/\y in {
	0/1/0/0, 1/1/0/0, 6/1/0/0,
	0/1/0/1, 1/-1/0/1, 6/-1/0/1,
	0/1/0/2, 1/-1/0/2, 6/-1/0/2,
	0/1/1/0, 3/1/1/0, 4/-1/1/0,
	0/1/1/1, 3/1/1/1, 4/-1/1/1,
	0/1/1/2, 3/1/1/2, 4/-1/1/2
	}
\fill[gray]
	([shift={(\u cm + 4.5*\x cm, -2.7*\y cm)}] 90-51.43*\u*\a:0.8) circle (0.05);

\foreach \a/\u/\y/\x in {
	2/-1/0/0, 2/1/0/0, 5/-1/0/0, 5/1/0/0,
	2/-1/1/0, 2/1/1/0, 5/-1/1/0, 5/1/1/0,
	2/-1/2/0, 2/1/2/0, 5/-1/2/0, 5/1/2/0,
	1/-1/0/1, 1/1/0/1, 2/-1/0/1, 2/1/0/1,
	1/-1/1/1, 1/1/1/1, 2/-1/1/1, 2/1/1/1,
	1/-1/2/1, 1/1/2/1, 5/-1/2/1, 5/1/2/1
	}
\fill[red]
	([shift={(\u cm + 4.5*\x cm, -2.7*\y cm)}] 90-51.43*\u*\a:0.8) circle (0.05);

\foreach \a/\u/\y/\x in {
	3/-1/0/0, 3/1/0/0, 4/-1/0/0, 4/1/0/0,
	3/-1/1/0, 3/1/1/0, 4/-1/1/0, 4/1/1/0,
	3/-1/2/0, 3/1/2/0, 4/-1/2/0, 4/1/2/0,
	5/-1/0/1, 5/1/0/1, 6/-1/0/1, 6/1/0/1,
	5/-1/1/1, 5/1/1/1, 6/-1/1/1, 6/1/1/1,
	2/-1/2/1, 2/1/2/1, 6/-1/2/1, 6/1/2/1
	}
\fill[blue]
	([shift={(\u cm + 4.5*\x cm, -2.7*\y cm)}] 90-51.43*\u*\a:0.8) circle (0.05);


\begin{scope}[xshift=4.5*2 cm]

\foreach \a in {0,...,7}
\foreach \u in {1,-1}
\foreach \x in {0,1}
\foreach \y in {0,1,2}
{
\begin{scope}[shift={(\u cm+4.5*\x cm, -2.7*\y cm)}]

\draw
	(-22.5+45*\a:0.8) -- (22.5+45*\a:0.8);
	
\node at (90+22.5*\u-45*\a*\u:0.62) {\tiny \a};

\end{scope}
}

\foreach \a in {0,...,7}
\foreach \x in {0,1,2}
\foreach \y in {0,1,2}
{

\coordinate (A\a\y\x) at ([shift={(-1 cm+4.5*\x cm, -2.7*\y cm)}] 90+45*\a:0.74);
\coordinate (B\a\y\x) at ([shift={(1 cm+4.5*\x cm, -2.7*\y cm)}]  90-45*\a:0.74);

}

\foreach \y in {0,1}
\draw[yshift=-2.7*\y cm]
	(A0\y 0) to[out=-90, in=45] (A3\y 0)
	(B0\y 0) to[out=-90, in=135] (B3\y 0)
	(A4\y 0) .. controls (-0.2,-1) and (0.2,-1) .. (B4\y 0)
	(A5\y 0) .. controls (0,-0.8) and (0,0.8) .. (B7\y 0)
	(B5\y 0) .. controls (0,-0.8) and (0,0.8) .. (A7\y 0);

\draw[yshift=-2.7*2 cm]
	(A120) .. controls (-1.6,1.3) and (1.6,1.3) .. (B120)
	(A520) .. controls (0,-0.8) and (0,0.8) .. (B720)
	(B520) .. controls (0,-0.8) and (0,0.8) .. (A720);

\draw[dash pattern=on 2pt off 1pt]
	(A600) -- (B600)
	(A100) to[out=-45, in=0] (A200)
	(B100) to[out=225, in=180] (B200);

\draw[dash pattern=on 2pt off 1pt, yshift=-2.7*1 cm]
	(A610) -- (B610)
	(A210) .. controls (-2.6,1.8) and (1.7,1.3) .. (B110)
	(B210) .. controls (2.6,1.8) and (-1.7,1.3) .. (A110);
		
\draw[dash pattern=on 2pt off 1pt, yshift=-2.7*2 cm]
	(A620) -- (B620)
	(A020) to[out=-90, in=45] (A320)
	(B020) to[out=-90, in=135] (B320)
	(A220) .. controls (-2.4,-1.3) and (0,-1.1) .. (B420)
	(B220) .. controls (2.4,-1.3) and (0,-1.1) .. (A420);

\begin{scope}[xshift=4.5*1 cm]

\draw
	(A601) -- (B601)
	(A201) .. controls (-2.4,-1.3) and (0,-1.1) .. (B401)
	(B201) .. controls (2.4,-1.3) and (0,-1.1) .. (A401);

\draw[yshift=-2.7*1 cm]
	(A611) -- (B611)
	(A111) to[out=-45, in=90] (A411)
	(B011) to[out=-90, in=135] (B311);

\draw[yshift=-2.7*2 cm]
	(A421) .. controls (-0.2,-1) and (0.2,-1) .. (B421)
	(A221) .. controls (-2.4,1.5) and (0,0.8) .. (B721)
	(B121) to[out=60, in=0] (1,1) to[out=180, in=60] (A621);

\draw[dash pattern=on 2pt off 1pt]
	(A101) to[out=-45, in=225] (A701)
	(B101) to[out=225, in=-45] (B701)
	(A301) .. controls (-1.6,-1.3) and (1.6,-1.3) .. (B301)
	(A501) .. controls (0.1,-0.8) and (0,1.4) .. (B001)
	(B501) .. controls (-0.1,-0.8) and (0,1.4) .. (A001);

\draw[dash pattern=on 2pt off 1pt, yshift=-2.7*1 cm]
	(A511) .. controls (0,-0.6) .. (B511)
	(A711) .. controls (0,0.6) .. (B711)
	(A011) .. controls (0.1,1.5) and (-0.1,-1.5) .. (B411)
	(A211) .. controls (-2.6,1.5) and (2.6,1.5) .. (B211)
	(B111) to[out=0, in=90] (2,-0.2) to[out=-90, in=-80] (A311);
		
\draw[dash pattern=on 2pt off 1pt, yshift=-2.7*2 cm]
	(A121) to[out=-45, in=225] (A721)
	(B221) -- (B621)
	(A021) .. controls (-0.2,1) and (0.2,1) .. (B021)
	(A521) .. controls (0,-0.6) .. (B521)
	(A321) .. controls (-1.6,-1.3) and (1.6,-1.3) .. (B321);
				
\end{scope}


\foreach \u in {1,-1}
{
\foreach \a/\y/\x in {
	0/0/0, 3/0/0, 0/1/0, 3/1/0, 0/2/0, 3/2/0,
	1/0/1, 7/0/1, 0/1/1, 1/1/1, 3/1/1, 4/1/1, 1/2/1, 2/2/1, 6/2/1, 7/2/1}
\draw[red, thick, shift={(\u cm+4.5*\x cm, -2.7*\y cm)}]
	(90-22.5*\u-45*\u*\a:0.8) -- (90+22.5*\u-45*\u*\a:0.8);

\foreach \a/\y/\x in {
	5/0/0, 7/0/0, 5/1/0, 7/1/0, 5/2/0, 7/2/0,
	0/0/1, 5/0/1 }
\draw[blue, thick, shift={(\u cm+4.5*\x cm, -2.7*\y cm)}]
	(90-22.5*\u-45*\u*\a:0.8) -- (90+22.5*\u-45*\u*\a:0.8);

\foreach \a/\y/\x in {
	1/0/0, 2/0/0, 1/1/0, 2/1/0, 2/2/0, 4/2/0, 2/0/1, 4/0/1}
\draw[green, shift={(\u cm+4.5*\x cm, -2.7*\y cm)}]
	(90-22.5*\u-45*\u*\a:0.8) -- (90+22.5*\u-45*\u*\a:0.8);
	
}


\foreach \a/\u/\y/\x in {
	0/-1/0/0, 4/-1/0/0, 5/1/0/0,
	0/-1/1/0, 4/-1/1/0, 5/1/1/0,
	0/-1/2/0, 3/-1/2/0, 5/1/2/0,
	0/-1/0/1, 2/-1/0/1, 5/1/0/1,
	0/-1/1/1, 4/1/1/1, 0/1/1/1,
	0/-1/2/1, 2/-1/2/1, 0/1/2/1}
\fill
	([shift={(\u cm+4.5*\x cm, -2.7*\y cm)}] 90+22.5*\u-45*\u*\a:0.8) circle (0.05);

\foreach \a/\u/\y/\x in {
	0/1/0/0, 4/1/0/0, 5/-1/0/0,
	0/1/1/0, 4/1/1/0, 5/-1/1/0,
	0/1/2/0, 3/1/2/0, 5/-1/2/0,
	0/1/0/1, 2/1/0/1, 5/-1/0/1,
	1/-1/1/1, 5/-1/1/1, 5/1/1/1,
	1/-1/2/1, 7/-1/2/1, 1/1/2/1}
\fill[gray]
	([shift={(\u cm+4.5*\x cm, -2.7*\y cm)}] 90+22.5*\u-45*\u*\a:0.8) circle (0.05);

\foreach \a/\u/\y/\x in {
	1/-1/0/0, 2/-1/0/0, 3/-1/0/0,
	1/-1/1/0, 2/1/1/0, 3/-1/1/0,
	1/-1/2/0, 2/1/2/0, 4/-1/2/0,
	1/-1/0/1, 7/-1/0/1, 6/1/0/1,
	2/-1/1/1, 4/-1/1/1, 2/1/1/1,
	3/-1/2/1, 3/1/2/1, 7/1/2/1}
\fill[red]
	([shift={(\u cm+4.5*\x cm, -2.7*\y cm)}] 90+22.5*\u-45*\u*\a:0.8) circle (0.05);

\foreach \a/\u/\y/\x in {
	1/1/0/0, 2/1/0/0, 3/1/0/0,
	1/1/1/0, 2/-1/1/0, 3/1/1/0,
	1/1/2/0, 2/-1/2/0, 4/1/2/0,
	1/1/0/1, 7/1/0/1, 6/-1/0/1,
	3/-1/1/1, 1/1/1/1, 3/1/1/1,
	6/-1/2/1, 6/1/2/1, 2/1/2/1}
\fill[blue]
	([shift={(\u cm+4.5*\x cm, -2.7*\y cm)}] 90+22.5*\u-45*\u*\a:0.8) circle (0.05);

\foreach \a/\u/\y/\x in {
	6/-1/0/0, 7/-1/0/0, 6/1/0/0, 7/1/0/0,
	6/-1/1/0, 7/-1/1/0, 6/1/1/0, 7/1/1/0,
	6/-1/2/0, 7/-1/2/0, 6/1/2/0, 7/1/2/0,
	3/-1/0/1, 4/-1/0/1, 3/1/0/1, 4/1/0/1,
	6/-1/1/1, 7/-1/1/1, 6/1/1/1, 7/1/1/1,
	4/-1/2/1, 5/-1/2/1, 4/1/2/1, 5/1/2/1}
\fill[green]
	([shift={(\u cm+4.5*\x cm, -2.7*\y cm)}] 90+22.5*\u-45*\u*\a:0.8) circle (0.05);

\end{scope}
			
\end{tikzpicture}
\caption{Tilings of $3{\bb P}^2$ by two heptagons or octagons, with 5 distinct edge lengths.} 
\label{P2B}
\end{figure}

\begin{figure}[htp]
\centering
\begin{tikzpicture}[>=latex,scale=0.85]

\draw[yshift=-3cm]
	(-2.2,-1.7) rectangle (6.7,1.4);


\foreach \a in {0,...,6}
\foreach \u in {1,-1}
\foreach \y in {0,1}
{
\begin{scope}[shift={(\u,-3*\y)}]

\draw
	(38.57+51.43*\a:0.8) -- (90+51.43*\a:0.8);
	
\node at (90-51.43*\a*\u:0.6) {\tiny \a};

\end{scope}
}

\foreach \a in {0,...,6}
\foreach \y in {0,1}
{
\coordinate (A\a\y) at ([shift={(-1,-3*\y)}] 115.71+51.43*\a:0.72);
\coordinate (B\a\y) at ([shift={(1,-3*\y)}]  64.28-51.43*\a:0.72);
}


\foreach \y in {0,1}
\draw[yshift=-3*\y cm]
	(A0\y) .. controls (-1.5,1.3) and (1.5,1.3) .. (B0\y)
	(A6\y) .. controls (-0.2,0.9) and (0.2,0.9) .. (B6\y);

\foreach \y in {0}
\draw[yshift=-3*\y cm]
	(A3\y) .. controls (-0.4,-1.1) and (0.4,-1.1) .. (B3\y)
	(A1\y) to[out=295.71+51.43*1, in=295.71+51.43*5] (A5\y)
	(B1\y) to[out=244.28-51.43*1, in=244.28-51.43*5] (B5\y);

\foreach \y in {1}
\draw[yshift=-3*\y cm]
	(A2\y) .. controls (-2,-1.4) and (2,-1.4) .. (B2\y)
	(A4\y) .. controls (0,-0.5) .. (B4\y);


\foreach \y in {0}
\draw[dash pattern=on 2pt off 1pt, yshift=-3*\y cm]
	(A2\y) .. controls (-2,-1.4) and (2,-1.4) .. (B2\y)
	(A4\y) .. controls (0,-0.5) .. (B4\y);

\foreach \y in {1}
\draw[dash pattern=on 2pt off 1pt, yshift=-3*\y cm]
	(A1\y) .. controls (-2.4,1.7) and (2.4,1.7) .. (B1\y)
	(A3\y) .. controls (-0.4,-1.1) and (0.4,-1.1) .. (B3\y)
	(A5\y) .. controls (0,0.2) .. (B5\y);


\foreach \a/\y in {1/0, 5/0}
\foreach \u in {1,-1}
\draw[red, thick]
	([shift={(\u,-3*\y)}] 90-51.43*\u*\a:0.8) -- 
	([shift={(\u,-3*\y)}] 90-51.43*\u*\a-51.43*\u:0.8);

	
\foreach \a/\u/\y in {
	0/-1/0, 6/1/0, 1/1/0,
	0/-1/1, 6/1/1, 1/1/1, 
	0/1/1, 6/-1/1, 1/-1/1
	}
\fill
	([shift={(\u,-3*\y)}] 90-51.43*\u*\a:0.8) circle (0.05);

\foreach \a/\u/\y in {
	0/1/0, 6/-1/0, 1/-1/0,
	2/1/1, 3/1/1, 2/-1/1, 3/-1/1
	}
\fill[red]
	([shift={(\u,-3*\y)}] 90-51.43*\u*\a:0.8) circle (0.05);

\foreach \a/\u/\y in {
	2/1/0, 5/1/0, 2/-1/0, 5/-1/0,
	4/1/1, 5/1/1, 4/-1/1, 5/-1/1
	}
\fill[blue]
	([shift={(\u,-3*\y)}] 90-51.43*\u*\a:0.8) circle (0.05);

\foreach \a/\u/\y in {
	3/1/0, 4/1/0, 3/-1/0, 4/-1/0
	}
\fill[green]
	([shift={(\u,-3*\y)}] 90-51.43*\u*\a:0.8) circle (0.05);

\node at (0,-1.4) {\footnotesize $3{\bb P}^2$, 7-gon, 6 lengths};
\node at (0,-4.4) {\footnotesize $4{\bb P}^2$, 7-gon, 7 lengths};


\begin{scope}[xshift=4.5cm]

\foreach \a in {0,...,7}
\foreach \u in {1,-1}
\foreach \y in {0,1}
{
\begin{scope}[shift={(\u,-3*\y)}]

\draw
	(-22.5+45*\a:0.8) -- (22.5+45*\a:0.8);
	
\node at (90+22.5*\u-45*\a*\u:0.62) {\tiny \a};

\end{scope}
}

\foreach \a in {0,...,7}
\foreach \y in {0,1}
{
\coordinate (A\a\y) at ([shift={(-1,-3*\y)}] 90+45*\a:0.74);
\coordinate (B\a\y) at ([shift={(1,-3*\y)}]  90-45*\a:0.74);
}


\draw
	(A00) .. controls (-0.2,1) and (0.2,1) .. (B00)
	(A10) .. controls (-1.6,1.3) and (1.6,1.3) .. (B10)
	(A30) .. controls (-1.6,-1.3) and (1.6,-1.3) .. (B30)
	(A20) to[out=0, in=225] (A70)
	(B20) to[out=180, in=-45] (B70)
	(A40) .. controls (-0.5,-1.3) and (0,-0.3) .. (B60)
	(B40) .. controls (0.5,-1.3) and (0,-0.3) .. (A60)
	(A61) -- (B61);

\draw[yshift=-3*1 cm]
	(A01) .. controls (-0.2,1) and (0.2,1) .. (B01)
	(A21) .. controls (-2.6,1.7) and (2.6,1.7) .. (B21)
	(A41) .. controls (-0.2,-1) and (0.2,-1) .. (B41)
	;


\draw[dash pattern=on 2pt off 1pt]
	(A50) .. controls (0,-0.6) .. (B50);
	
\draw[dash pattern=on 2pt off 1pt, yshift=-3*1 cm]
	(A11) .. controls (-1.6,1.3) and (1.6,1.3) .. (B11)
	(A51) .. controls (0,-0.6) .. (B51)
	(A71) .. controls (0,0.6) .. (B71)
	(A31) .. controls (-1.6,-1.3) and (1.6,-1.3) .. (B31);


\foreach \u in {1,-1}
\foreach \a in {4,6}
\draw[red, thick, xshift=\u cm]
	(90-22.5*\u-45*\u*\a:0.8) -- (90+22.5*\u-45*\u*\a:0.8);

\foreach \u in {1,-1}
\foreach \a in {2, 7}
\draw[blue, thick, xshift=\u cm]
	(90-22.5*\u-45*\u*\a:0.8) -- (90+22.5*\u-45*\u*\a:0.8);
	
	
\foreach \a/\u/\y in {
	0/-1/0, 2/-1/0, 1/1/0,
	0/-1/1, 0/1/1, 1/-1/1, 1/1/1}
\fill
	([shift={(\u,-3*\y)}] 90+22.5*\u-45*\u*\a:0.8) circle (0.05);

\foreach \a/\u/\y in {
	0/1/0, 2/1/0, 1/-1/0,
	4/-1/1, 4/1/1, 5/-1/1, 5/1/1}
\filldraw[fill=white]
	([shift={(\u,-3*\y)}] 90+22.5*\u-45*\u*\a:0.8) circle (0.05);

\foreach \a/\u/\y in {
	3/-1/0, 7/-1/0, 4/1/0,
	2/-1/1, 2/1/1, 3/-1/1, 3/1/1}
\fill[red]
	([shift={(\u,-3*\y)}] 90+22.5*\u-45*\u*\a:0.8) circle (0.05);

\foreach \a/\u/\y in {
	3/1/0, 7/1/0, 4/-1/0,
	6/-1/1, 6/1/1, 7/-1/1, 7/1/1}
\fill[blue]
	([shift={(\u,-3*\y)}] 90+22.5*\u-45*\u*\a:0.8) circle (0.05);

\foreach \a/\u/\y in {
	5/1/0, 6/1/0, 5/-1/0, 6/-1/0}
\fill[green]
	([shift={(\u,-3*\y)}] 90+22.5*\u-45*\u*\a:0.8) circle (0.05);

\node at (0,-1.4) {\footnotesize $3{\bb P}^2$, 8-gon, 6 lengths};
\node at (0,-4.4) {\footnotesize $4{\bb P}^2$, 8-gon, 8 lengths};
							
\end{scope}


\begin{scope}[shift={(4.5,-6)}]

\foreach \a in {0,...,8}
\foreach \u in {1,-1}
{
\begin{scope}[xshift=\u cm]

\draw
	(90+40*\a:0.8) -- (50+40*\a:0.8);
	
\node at (90-40*\a*\u:0.6) {\tiny \a};

\end{scope}
}

\foreach \a in {0,...,8}
{
\coordinate (A\a) at ([xshift=-1cm] 110+40*\a:0.75);
\coordinate (B\a) at ([xshift=1cm] 70-40*\a:0.75);
}

\draw
	(A1) to[out=-30, in=210] (A7)
	(B1) to[out=210, in=-30] (B7)
	(A0) .. controls (-1,1.2) and (1,1.2) .. (B0)
	(A8) .. controls (-0.4,0.9) and (0.4,0.9) .. (B8)
	(A5) .. controls (-0.2,-0.7) and (0.2,-0.7) .. (B5)
	(A3) .. controls (-1.4,-1.3) and (1.4,-1.3) .. (B3);

\draw[dash pattern=on 2pt off 1pt]
	(A6) to[out=-20, in=200] (B6)
	(A4) .. controls (-0.2,-1) and (0.2,-1) .. (B4)
	(A2) .. controls (-2.6,-1.7) and (2.6,-1.7) .. (B2);


\foreach \u in {-1,1}
\draw[red, thick, xscale=\u, xshift=-1cm]
	(50:0.8) -- (10:0.8)
	(130:0.8) -- (170:0.8);

	
\foreach \a/\u in {
	0/-1, 8/1, 1/1}
\fill[xshift=\u cm]
	(90-40*\u*\a:0.8) circle (0.05);

\foreach \a/\u in {
	0/1, 8/-1, 1/-1}
\filldraw[fill=white, xshift=\u cm]
	(90-40*\u*\a:0.8) circle (0.05);

\foreach \a/\u in {
	2/1, 2/-1, 7/1, 7/-1}
\fill[red, xshift=\u cm]
	(90-40*\u*\a:0.8) circle (0.05);

\foreach \a/\u in {
	3/1, 3/-1, 4/1, 4/-1}
\fill[blue, xshift=\u cm]
	(90-40*\u*\a:0.8) circle (0.05);

\foreach \a/\u in {
	5/1, 5/-1, 6/1, 6/-1}
\fill[green, xshift=\u cm]
	(90-40*\u*\a:0.8) circle (0.05);

\node at (0,-1.45) {\footnotesize $4{\bb P}^2$, 9-gon, 8 lengths};
				
\end{scope}


\begin{scope}[xshift=4.5*2 cm]

\foreach \a in {0,...,9}
\foreach \u in {1,-1}
\foreach \y in {0,1}
{
\begin{scope}[shift={(\u,-3*\y)}]

\draw
	(72+36*\a:0.8) -- (108+36*\a:0.8);
	
\node at (90+18*\u-36*\a*\u:0.63) {\tiny \a};

\end{scope}
}

\foreach \a in {0,...,9}
\foreach \y in {0,1}
{
\coordinate (A\a\y) at ([shift={(-1,-3*\y)}] 90+36*\a:0.76);
\coordinate (B\a\y) at ([shift={(1,-3*\y)}] 90-36*\a:0.76);
}


\draw
	(A10) .. controls (-1.3,1.3) and (1.3,1.3) .. (B10)
	(A20) .. controls (-2.6,1.7) and (2.6,1.7) .. (B20)
	(A40) .. controls (-1.6,-1.3) and (1.6,-1.3) .. (B40)
	(A70) to[out=-18, in=198] (B70)	
	(A00) to[out=-90, in=18] (A30)
	(B00) to[out=-90, in=162] (B30)
	(A50) .. controls (0,-1.3) and (0,0.9) .. (B90)
	(B50) .. controls (0,-1.3) and (0,0.9) .. (A90);

\draw[yshift=-3cm]
	(A11) .. controls (-1.3,1.3) and (1.3,1.3) .. (B11)
	(A21) .. controls (-2.6,1.7) and (2.6,1.7) .. (B21)
	(A61) .. controls (-0.2,-0.8) and (0.2,-0.8) .. (B61)
	(A71) to[out=-18, in=198] (B71)	
	(A01) to[out=-90, in=18] (A31)
	(B01) to[out=-90, in=162] (B31)	
	(A51) to[out=90, in=198] (A81)
	(B51) to[out=90, in=-18] (B81);


\draw[dash pattern=on 2pt off 1pt]
	(A60) .. controls (-0.2,-0.8) and (0.2,-0.8) .. (B60)
	(A80) to[out=18, in=162] (B80);

\draw[dash pattern=on 2pt off 1pt, yshift=-3 cm]
	(A91) .. controls (-0.2,0.8) and (0.2,0.8) .. (B91)
	(A41) .. controls (-1.6,-1.3) and (1.6,-1.3) .. (B41);


\foreach \u in {1,-1}
{
\draw[blue, thick, xscale=\u, xshift=-1cm]
	(36:0.8) -- (72:0.8)
	(-72:0.8) -- (-108:0.8);
\draw[red, xscale=\u, xshift=-1cm]
	(180:0.8) -- (216:0.8)
	(72:0.8) -- (108:0.8);

\draw[blue, thick, yshift=-3cm, xscale=\u, xshift=-1cm]
	(36:0.8) -- (0:0.8)
	(-72:0.8) -- (-108:0.8);
\draw[red, thick, yshift=-3cm, xscale=\u, xshift=-1cm]
	(180:0.8) -- (216:0.8)
	(72:0.8) -- (108:0.8);	
}

	
\foreach \a/\u/\y in {
	0/-1/0, 4/-1/0, 5/1/0,
	0/-1/1, 4/-1/1, 0/1/1, 4/1/1}
\fill[shift={(\u,-3*\y)}]
	(90+18*\u-36*\u*\a:0.8) circle (0.05);

\foreach \a/\u/\y in {
	0/1/0, 4/1/0, 5/-1/0,
	5/-1/1, 9/-1/1, 5/1/1, 9/1/1}
\filldraw[fill=white, shift={(\u,-3*\y)}]
	(90+18*\u-36*\u*\a:0.8) circle (0.05);

\foreach \a/\u/\y in {
	1/-1/0, 2/1/0, 3/-1/0,
	1/-1/1, 2/1/1, 3/-1/1}
\fill[red, shift={(\u,-3*\y)}]
	(90+18*\u-36*\u*\a:0.8) circle (0.05);

\foreach \a/\u/\y in {
	1/1/0, 2/-1/0, 3/1/0,
	1/1/1, 2/-1/1, 3/1/1}
\fill[blue, shift={(\u,-3*\y)}]
	(90+18*\u-36*\u*\a:0.8) circle (0.05);

\foreach \a/\u/\y in {
	6/1/0, 6/-1/0, 9/1/0, 9/-1/0,
	6/-1/1, 7/1/1, 8/-1/1}
\fill[green, shift={(\u,-3*\y)}]
	(90+18*\u-36*\u*\a:0.8) circle (0.05);

\foreach \a/\u/\y in {
	7/1/0, 7/-1/0, 8/1/0, 8/-1/0,
	6/1/1, 7/-1/1, 8/1/1}
\fill[brown, shift={(\u,-3*\y)}]
	(90+18*\u-36*\u*\a:0.8) circle (0.05);

\foreach \y in {0,1}
\node at (0,-1.4-3*\y) {\footnotesize $4{\bb P}^2$, 10-gon, 8 lengths};
						
\end{scope}


\begin{scope}[xshift=3*4.5 cm]

\foreach \a in {0,...,10}
\foreach \u in {1,-1}
\foreach \y in {0,1,2}
{
\begin{scope}[shift={(\u,-3*\y)}]

\draw
	(90+32.73*\a:0.8) -- (122.73+32.73*\a:0.8);
	
\node at (90-32.73*\a*\u:0.65) {\tiny \a};

\end{scope}
}

\foreach \a in {0,...,10}
\foreach \y in {0,1,2}
{
\coordinate (A\a\y) at ([shift={(-1,-3*\y)}] 106.36+32.73*\a:0.77);
\coordinate (B\a\y) at ([shift={(1,-3*\y)}] 73.64-32.73*\a:0.77);
}

\foreach \a/\b/\y in {
	1/9/0, 4/6/0,
	1/9/1, 5/7/1,
	0/3/2, 7/10/2, 4/6/2}
\draw
	(A\a\y) to[out=-73.64+32.73*\a, in=-73.64+32.73*\b] (A\b\y)
	(B\a\y) to[out=-106.36-32.73*\a, in=-106.36-32.73*\b] (B\b\y);

\draw
	(A00) .. controls (-1,1.1) and (1,1.1) .. (B00)
	(A100) .. controls (-0.5,0.9) and (0.5,0.9) .. (B100)
	(A20) .. controls (-2.4,1.5) and (2.4,1.5) .. (B20)
	(A70) .. controls (0,-0.4) .. (B70)
	(A30) to[out=220, in=180] (-0.9,-1.1) to[out=0, in=220] (B80)
	(B30) to[out=-40, in=0] (0.9,-1.1) to[out=180, in=-40] (A80);

\draw[yshift=-3cm]
	(A01) .. controls (-1,1.1) and (1,1.1) .. (B01)
	(A101) .. controls (-0.5,0.9) and (0.5,0.9) .. (B101)
	(A21) .. controls (-2.4,1.5) and (2.4,1.5) .. (B21)
	(A41) .. controls (-1.2,-1.3) and (1.2,-1.3) .. (B41)
	(A31) to[out=220, in=180] (-0.9,-1.1) to[out=0, in=220] (B81)
	(B31) to[out=-40, in=0] (0.9,-1.1) to[out=180, in=-40] (A81);

\draw[yshift=-6cm]
	(A12) .. controls (-1.8,1.3) and (1.8,1.3) .. (B12)
	(A92) .. controls (-0.2,0.6) and (0.2,0.6) .. (B92)
	(A82) .. controls (0,0.15) .. (B82)
	(A22) .. controls (-2.6,1.7) and (2.6,1.7) .. (B22);
	
\draw[dash pattern=on 2pt off 1pt]
	(A50) .. controls (-0.2,-1) and (0.2,-1) .. (B50);

\draw[dash pattern=on 2pt off 1pt, yshift=-3 cm]
	(A61) .. controls (-0.4,-0.8) and (0.4,-0.8) .. (B61);
	
\draw[dash pattern=on 2pt off 1pt, yshift=-6 cm]
	(A52) .. controls (-0.2,-1) and (0.2,-1) .. (B52);


\foreach \u in {1,-1}
{
\foreach \a/\y in {1/0, 9/0, 1/1, 9/1, 0/2, 3/2}
\draw[red, thick, yshift=-3*\y cm, xscale=\u, xshift=-1cm]
	(90+32.73*\a:0.8) -- (122.73+32.73*\a:0.8);
	
\foreach \a/\y in {3/0, 8/0, 3/1, 8/1, 7/2, 10/2}
\draw[blue, thick, yshift=-3*\y cm, xscale=\u, xshift=-1cm]
	(90+32.73*\a:0.8) -- (122.73+32.73*\a:0.8);

\foreach \a/\y in {4/0, 6/0, 5/1, 7/1, 4/2, 6/2}
\draw[green, thick, yshift=-3*\y cm, xscale=\u, xshift=-1cm]
	(90+32.73*\a:0.8) -- (122.73+32.73*\a:0.8);
}

	
\foreach \a/\u/\y in {
	0/-1/0, 1/1/0, 10/1/0,
	0/-1/1, 1/1/1, 10/1/1,
	0/-1/2, 4/-1/2, 7/-1/2}
\fill[shift={(\u,-3*\y)}]
	(90-32.73*\u*\a:0.8) circle (0.05);

\foreach \a/\u/\y in {
	0/1/0, 1/-1/0, 10/-1/0,
	0/1/1, 1/-1/1, 10/-1/1,
	0/1/2, 4/1/2, 7/1/2}
\filldraw[fill=white, shift={(\u,-3*\y)}]
	(90-32.73*\u*\a:0.8) circle (0.05);

\foreach \a/\u/\y in {
	2/-1/0, 9/-1/0, 3/1/0,
	2/-1/1, 9/-1/1, 3/1/1,
	1/-1/2, 3/-1/2, 2/1/2}
\fill[red, shift={(\u,-3*\y)}]
	(90-32.73*\u*\a:0.8) circle (0.05);

\foreach \a/\u/\y in {
	2/1/0, 9/1/0, 3/-1/0,
	2/1/1, 9/1/1, 3/-1/1,
	1/1/2, 3/1/2, 2/-1/2}
\fill[blue, shift={(\u,-3*\y)}]
	(90-32.73*\u*\a:0.8) circle (0.05);

\foreach \a/\u/\y in {
	4/-1/0, 7/-1/0, 8/1/0, 
	5/-1/1, 8/-1/1, 4/1/1,
	8/-1/2, 10/-1/2, 9/1/2}
\fill[green, shift={(\u,-3*\y)}]
	(90-32.73*\u*\a:0.8) circle (0.05);

\foreach \a/\u/\y in {
	4/1/0, 7/1/0, 8/-1/0,
	5/1/1, 8/1/1, 4/-1/1,
	8/1/2, 10/1/2, 9/-1/2}
\fill[brown, shift={(\u,-3*\y)}]
	(90-32.73*\u*\a:0.8) circle (0.05);
		
\foreach \a/\u/\y in {
	5/-1/0, 6/-1/0, 5/1/0, 6/1/0,
	7/-1/1, 6/-1/1, 7/1/1, 6/1/1,
	5/-1/2, 6/-1/2, 5/1/2, 6/1/2}
\fill[orange, shift={(\u,-3*\y)}]
	(90-32.73*\u*\a:0.8) circle (0.05);

\foreach \y in {0,1,2}
\node at (0,-1.4-3*\y) {\footnotesize $4{\bb P}^2$, 11-gon, 8 lengths};
		
\end{scope}

\end{tikzpicture}
\caption{Tilings of $3{\bb P}^2$ and $4{\bb P}^2$ by two congruent polygons, with maximal number of distinct edge lengths. Not included: $4{\bb P}^2$, 12-gon, 8 lengths.} 
\label{P2A}
\end{figure}
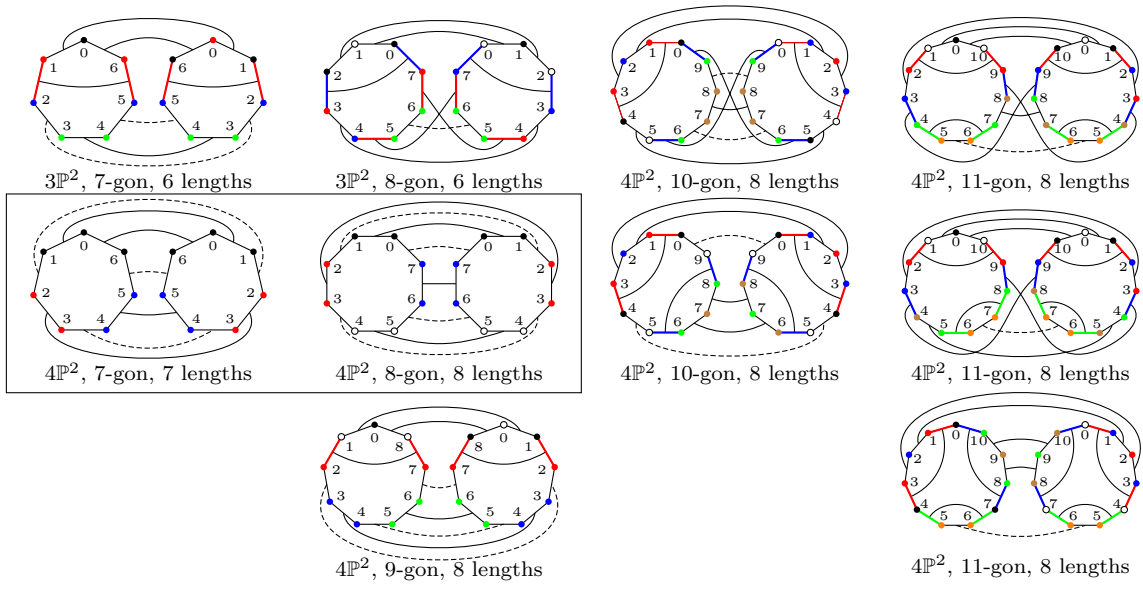

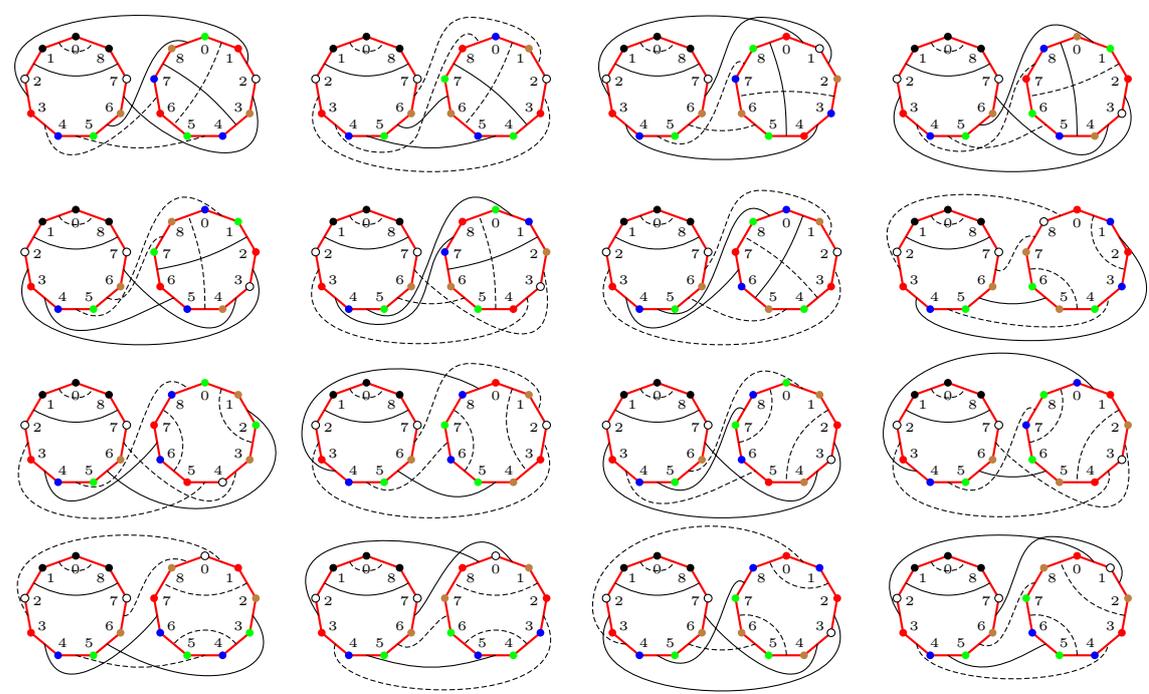
\begin{figure}[htp]
\centering
\begin{tikzpicture}[>=latex,scale=0.85]

\foreach \a in {0,...,8}
\foreach \u in {1,-1}
\foreach \x in {0,...,3}
\foreach \y in {0,...,3}
{
\begin{scope}[shift={(\u cm+4.5*\x cm, -2.7*\y cm)}]

\draw[red, thick]
	(90+40*\a:0.8) -- (50+40*\a:0.8);
	
\node at (90-40*\a*\u:0.6) {\tiny \a};

\end{scope}
}

\foreach \a in {0,...,8}
\foreach \x in {0,...,3}
\foreach \y in {0,...,3}
{
\coordinate (A\a\x\y) at ([shift={(-1 cm+4.5*\x cm, -2.7*\y cm)}] 110+40*\a:0.75);
\coordinate (B\a\x\y) at ([shift={(1 cm+4.5*\x cm, -2.7*\y cm)}]  70-40*\a:0.75);
}

\foreach \x in {0,...,3}
\foreach \y in {0,...,3}
{
\draw
	(A1\x\y) to[out=-30, in=210] (A7\x\y);
\draw[dash pattern=on 2pt off 1pt]
	(A0\x\y) to[out=-70, in=250] (A8\x\y);
}

\foreach \a/\b/\x/\y in {
	3/7/0/0, 3/7/1/0, 4/8/2/0, 4/8/3/0, 1/6/0/1, 1/6/1/1, 0/5/2/1}
\draw
	(B\a\x\y) to[out=-110-40*\a, in=-110-40*\b] (B\b\x\y);

\foreach \a/\b/\x/\y in {
	0/5/0/0, 0/5/1/0, 2/6/2/0, 1/6/3/0, 4/8/0/1, 4/8/1/1, 3/7/2/1, 0/2/3/1, 4/6/3/1, 0/2/0/2, 5/7/0/2, 0/3/1/2, 5/7/1/2, 1/4/2/2, 6/8/2/2, 1/4/3/2, 6/8/3/2, 1/7/0/3, 3/5/0/3, 1/7/1/3, 3/5/1/3, 1/8/2/3, 4/6/2/3, 2/8/3/3, 4/6/3/3}
\draw[dash pattern=on 2pt off 1pt]
	(B\a\x\y) to[out=-110-40*\a, in=-110-40*\b] (B\b\x\y);
	

\draw
	(A500) .. controls (0,-0.8) and (0,1) .. (B800)
	(A600) .. controls (0.5,-1.1) and (2.2,-1.5) .. (B200)
	(A200) .. controls (-3,1.4) and (2,1.5) .. (B100);

\draw[dash pattern=on 2pt off 1pt]
	(A400) .. controls (-0.2,-1) and (0.2,-1) .. (B400)
	(A300) .. controls (-1.3,-1.7) and (0,-0.5) .. (B600);


\begin{scope}[shift={(4.5*1 cm, -2.7*0 cm)}]

\draw
	(A410) .. controls (-0.2,-1) and (0.2,-1) .. (B410)
	(A510) .. controls (-0.2,-0.8) and (0,-0.2)  .. (B610);

\draw[dash pattern=on 2pt off 1pt]
	(A210) .. controls (-2.6,-1.7) and (2.6,-1.7) .. (B210)
	(A610) to[out=50, in=180] (0.6,1.1) to[out=0, in=60] (B110)
	(A310) to[out=-80, in=-110] (-0.2,-0.7) to[out=70, in=120] (B810);
	
\end{scope}


\begin{scope}[shift={(4.5*2 cm, -2.7*0 cm)}]

\draw
	(A620) to[out=50, in=180] (0.6,1.1) to[out=0, in=60] (B120)
	(A220) .. controls (-2.7,1.3) and (0.8,1.4) .. (B020)
	(A320) .. controls (-1.4,-1.3) and (1.4,-1.3) .. (B320);

\draw[dash pattern=on 2pt off 1pt]
	(A520) .. controls (-0.2,-0.7) and (0.2,-0.7) .. (B520)
	(A420) .. controls (0,-1.4) and (0,0.7) .. (B720);
	
\end{scope}


\begin{scope}[shift={(4.5*3 cm, -2.7*0 cm)}]

\draw
	(A230) .. controls (-2.6,-1.7) and (2.6,-1.7) .. (B230)
	(A630) .. controls (0,-0.5) and (1.3,-1.7) .. (B330)
	(A530) .. controls (0.2,-0.7) and (-0.1,1.8) .. (B030);

\draw[dash pattern=on 2pt off 1pt]
	(A330) to[out=-80, in=-120] (-0.2,-0.7) to[out=60, in=210] (B730)
	(A430) .. controls (-0.4,-1.2) and (0.2,-0.7) .. (B530);
		
\end{scope}


\begin{scope}[shift={(4.5*0 cm, -2.7*1 cm)}]

\draw
	(A201) .. controls (-2.6,-1.7) and (2.6,-1.7) .. (B201)
	(A301) .. controls (-1.3,-1.6) and (0,-0.8) .. (B501)
	(A601) .. controls (0,-0.5) and (1.3,-1.7) .. (B301);
	
\draw[dash pattern=on 2pt off 1pt]
	(A401) .. controls (0.2,-1.5) and (-0.2,2) .. (B001)
	(A501) .. controls (0,-0.8) and (0,0.4) .. (B701);
	
\end{scope}


\begin{scope}[shift={(4.5*1 cm, -2.7*1 cm)}]

\draw
	(A311) to[out=-80, in=-120] (-0.2,-0.7) to[out=60, in=210] (B711)
	(A411) .. controls (0.2,-1.5) and (-0.2,2) .. (B011);
	
\draw[dash pattern=on 2pt off 1pt]
	(A511) .. controls (-0.2,-0.7) and (0.2,-0.7) .. (B511)
	(A211) .. controls (-2.6,-1.7) and (1.8,-1.5) .. (B311)
	(A611) .. controls (0,-0.5) and (2.2,-2.2) .. (B211);	
	
\end{scope}


\begin{scope}[shift={(4.5*2 cm, -2.7*1 cm)}]

\draw
	(A321) .. controls (-1.3,-1.7) and (0,-0.5) .. (B621)
	(A421) .. controls (0.1,-1.3) and (-0.1,1.4) .. (B821);

\draw[dash pattern=on 2pt off 1pt]
	(A221) .. controls (-2.6,-1.7) and (2.6,-1.7) .. (B221)
	(A621) to[out=50, in=180] (0.6,1.1) to[out=0, in=60] (B121)
	(B421) .. controls (0.4,-1.2) and (-0.2,-0.7) .. (A521); 	
	
\end{scope}


\begin{scope}[shift={(4.5*3 cm, -2.7*1 cm)}]

\draw
	(A531) .. controls (-0.2,-0.7) and (0.2,-0.7) .. (B531)
	(A331) .. controls (-1.8,-1.5) and (3.6,-1.7) .. (B131); 	
	
\draw[dash pattern=on 2pt off 1pt]	
	(A431) .. controls (-0.5,-1) and (1.5,-1.3) .. (B331)
	(A231) .. controls (-2.7,1.3) and (0,1.2) .. (B831)
	(A631) .. controls (0,-0.3) and (0,0.5) .. (B731);
	
\end{scope}


\begin{scope}[shift={(4.5*0 cm, -2.7*2 cm)}]

\draw
	(A302) .. controls (-1.3,-1.7) and (0,-0.5) .. (B602)
	(A502) .. controls (1,-2) and (3,-0.5) .. (B102); 

\draw[dash pattern=on 2pt off 1pt]
	(A602) .. controls (0,-0.5) and (1.3,-1.7) .. (B302)
	(A402) .. controls (0.1,-1.3) and (-0.1,1.4) .. (B802)
	(A202) .. controls (-2.6,-1.7) and (0.5,-1.5) .. (B402); 
		
\end{scope}


\begin{scope}[shift={(4.5*1 cm, -2.7*2 cm)}]

\draw
	(A312) to[out=180, in=-90] (-2,0) to[out=90, in=150] (B812)
	(A512) .. controls (0,-0.9) and (0.4,-1.2) .. (B412);
	
\draw[dash pattern=on 2pt off 1pt]
	(A612) to[out=50, in=180] (0.6,1.1) to[out=0, in=60] (B112)
	(A212) .. controls (-2.6,-1.7) and (2.6,-1.7) .. (B212)
	(A412) .. controls (-0.4,-1.2) and (0,-0.3) .. (B612); 
		
\end{scope}


\begin{scope}[shift={(4.5*2 cm, -2.7*2 cm)}]

\draw
	(A222) .. controls (-2.6,-1.7) and (2.6,-1.7) .. (B222)
	(A422) .. controls (0,-1.4) and (0,0.7) .. (B722)
	(A622) .. controls (0,-0.5) and (1.3,-1.7) .. (B322);

\draw[dash pattern=on 2pt off 1pt]
	(A322) .. controls (-1.3,-1.6) and (0,-0.8) .. (B522)
	(A522) .. controls (0.2,-0.7) and (-0.1,1.8) .. (B022);
			
\end{scope}


\begin{scope}[shift={(4.5*3 cm, -2.7*2 cm)}]

\draw
	(A332) to[out=180, in=-90] (-2,0) to[out=90, in=130] (B032)
	(A532) .. controls (-0.2,-0.7) and (0.2,-0.7) .. (B532);
	
\draw[dash pattern=on 2pt off 1pt]
	(A432) .. controls (0,-1.4) and (0,0.7) .. (B732)
	(A632) .. controls (0,-0.5) and (2.2,-2.2) .. (B232)
	(A232) .. controls (-2.6,-1.7) and (1.8,-1.5) .. (B332);
		
\end{scope}


\begin{scope}[shift={(4.5*0 cm, -2.7*3 cm)}]

\draw
	(A303) .. controls (-1.3,-1.7) and (0,-0.5) .. (B603)
	(A503) .. controls (1,-1.5) and (2.4,-1) .. (B203);

\draw[dash pattern=on 2pt off 1pt]
	(A403) .. controls (-0.2,-1) and (0.2,-1) .. (B403)
	(A203) .. controls (-2.7,1.3) and (0.8,1.4) .. (B003)
	(A603) .. controls (0,-0.2) and (0,1) .. (B803);
		
\end{scope}


\begin{scope}[shift={(4.5*1 cm, -2.7*3 cm)}]

\draw
	(A413) .. controls (-0.2,-1) and (0.2,-1) .. (B413)
	(A213) .. controls (-2.7,1.3) and (0,1.2) .. (B813)
	(A613) .. controls (0,0) and (0.5,1.7) .. (B013);
	
\draw[dash pattern=on 2pt off 1pt]
	(A513) .. controls (-0.2,-0.8) and (0,-0.2)  .. (B613)
	(A313) .. controls (-1.8,-1.5) and (2.6,-1.7) .. (B213);

\end{scope}


\begin{scope}[shift={(4.5*2 cm, -2.7*3 cm)}]

\draw
	(A223) .. controls (-2.6,-1.7) and (2.6,-1.7) .. (B223)
	(A423) .. controls (0,-1.4) and (0,0.7) .. (B723)
	(A623) .. controls (0,-0.5) and (1.3,-1.7) .. (B323);

\draw[dash pattern=on 2pt off 1pt]
	(A523) .. controls (-0.2,-0.7) and (0.2,-0.7) .. (B523)
	(A323) to[out=180, in=-90] (-2,0) to[out=90, in=130] (B023);
		
\end{scope}


\begin{scope}[shift={(4.5*3 cm, -2.7*3 cm)}]

\draw
	(A633) to[out=50, in=180] (0.6,1.1) to[out=0, in=60] (B133)
	(A233) .. controls (-2.7,1.3) and (0.8,1.4) .. (B033)
	(A433) .. controls (-0.4,-1.2) and (0.2,-0.7) .. (B533); 

\draw[dash pattern=on 2pt off 1pt]
	(A333) .. controls (-1.4,-1.3) and (1.4,-1.3) .. (B333)
	(A533) .. controls (0,-0.8) and (0,0.4) .. (B733);

\end{scope}


\foreach \x in {0,...,3}
\foreach \y in {0,...,3}
{
\begin{scope}[shift={(-1 cm +4.5*\x cm, -2.7*\y cm)}]

\foreach \a in {0,1,8}
\fill
	(90+40*\a:0.8) circle (0.06);

\foreach \a in {2,7}
\filldraw[fill=white]
	(90+40*\a:0.8) circle (0.06);

\fill[red]
	(90+40*3:0.8) circle (0.06);
\fill[blue]
	(90+40*4:0.8) circle (0.06);
\fill[green]
	(90+40*5:0.8) circle (0.06);
\fill[brown]
	(90+40*6:0.8) circle (0.06);
		
\end{scope}
}

\foreach \a/\x/\y in {
	2/0/0, 2/1/0, 1/2/0, 3/3/0, 
	3/0/1, 3/1/1, 2/2/1, 8/3/1, 
	4/0/2, 2/1/2, 3/2/2, 3/3/2, 
	0/0/3, 0/1/3, 3/2/3, 1/3/3 }
\filldraw[fill=white, shift={(1 cm +4.5*\x cm, -2.7*\y cm)}]
	(90-40*\a:0.8) circle (0.06);

\foreach \a/\x/\y in {
	1/0/0, 6/0/0,
	3/1/0, 8/1/0,
	0/2/0, 4/2/0,
	2/3/0, 7/3/0,
	2/0/1, 6/0/1,
	4/1/1, 8/1/1,
	3/2/1, 7/2/1,
	0/3/1, 2/3/1,
	5/0/2, 7/0/2,
	0/1/2, 3/1/2,
	2/2/2, 5/2/2,
	1/3/2, 4/3/2,
	1/0/3, 7/0/3,
	2/1/3, 8/1/3,
	0/2/3, 2/2/3,
	0/3/3, 3/3/3}
\fill[red, shift={(1 cm +4.5*\x cm, -2.7*\y cm)}]
	(90-40*\a:0.8) circle (0.06);

\foreach \a/\x/\y in {
	4/0/0, 7/0/0,
	0/1/0, 5/1/0,
	3/2/0, 7/2/0,
	5/3/0, 8/3/0,
	0/0/1, 5/0/1,
	1/1/1, 7/1/1,
	0/2/1, 6/2/1,
	1/3/1, 3/3/1,
	6/0/2, 8/0/2,
	6/1/2, 8/1/2,
	6/2/2, 8/2/2,
	0/3/2, 7/3/2,
	4/0/3, 6/0/3,
	3/1/3, 5/1/3,
	1/2/3, 8/2/3,
	4/3/3, 6/3/3}
\fill[blue, shift={(1 cm +4.5*\x cm, -2.7*\y cm)}]
	(90-40*\a:0.8) circle (0.06);

\foreach \a/\x/\y in {
	0/0/0, 5/0/0,
	4/1/0, 7/1/0,
	5/2/0, 8/2/0,
	1/3/0, 6/3/0,
	1/0/1, 7/0/1,
	0/1/1, 5/1/1,
	4/2/1, 8/2/1,
	4/3/1, 6/3/1,
	0/0/2, 2/0/2,
	5/1/2, 7/1/2,
	0/2/2, 7/2/2,
	6/3/2, 8/3/2,
	3/0/3, 5/0/3,
	4/1/3, 6/1/3,
	5/2/3, 7/2/3,
	5/3/3, 7/3/3}
\fill[green, shift={(1 cm +4.5*\x cm, -2.7*\y cm)}]
	(90-40*\a:0.8) circle (0.06);

\foreach \a/\x/\y in {
	3/0/0, 8/0/0,
	1/1/0, 6/1/0,
	2/2/0, 6/2/0,
	0/3/0, 4/3/0,
	4/0/1, 8/0/1,
	2/1/1, 6/1/1,
	1/2/1, 5/2/1,
	5/3/1, 7/3/1,
	1/0/2, 3/0/2,
	1/1/2, 4/1/2,
	1/2/2, 4/2/2,
	2/3/2, 5/3/2,
	2/0/3, 8/0/3,
	1/1/3, 7/1/3,
	4/2/3, 6/2/3,
	2/3/3, 8/3/3}
\fill[brown, shift={(1 cm +4.5*\x cm, -2.7*\y cm)}]
	(90-40*\a:0.8) circle (0.06);

\end{tikzpicture}
\caption{Tilings of $3{\bb P}^2$ by two congruent equilateral 9-gons.} 
\label{P2C}
\end{figure}

The algorithm above uses very primitive geometrical condition of the existence of positive angle solutions. The actual existence of prototiles as polygons with straight line edges requires further investigation. We do not have general results for the existence of the prototile except that some cases can be obviously realised by regular polygons. Even in such cases, we wish to have non-regular geometrical realisations that is not shared with the other tilings. For individual tilings, we may follow the idea in \cite{lwwy1} to show the geometrical realisability by actually constructing specific examples. We have verified the geometrical existence for all the tilings in Figures \ref{T2A}, \ref{P2B}, \ref{P2A}, \ref{P2C}, and we show two sets in Figures \ref{T2Ageom} and \ref{P2Cgeom}.

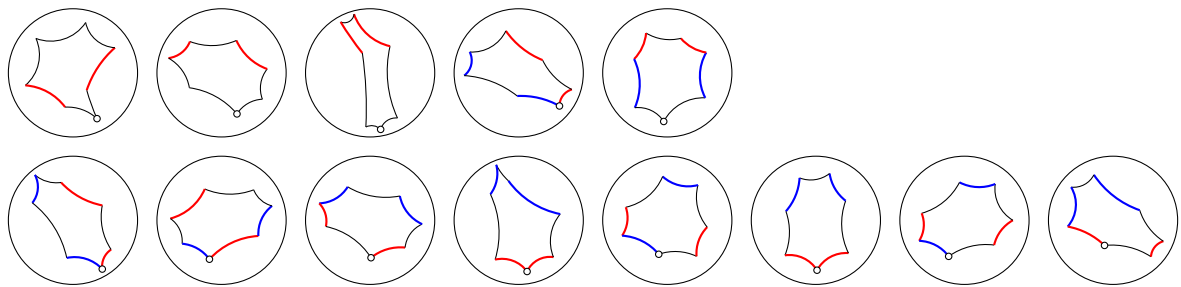
\begin{figure}[htp]
\centering
\begin{tikzpicture}[>=latex,scale=0.85]

\foreach \x in {0,...,4}
\draw
	(2.3*\x, 2.3) circle (1);
		
\foreach \x in {0,...,7}
\draw
	(2.3*\x, 0) circle (1);	


\begin{scope}[shift={(0*2.3,2.3)}]

\drawhypgeodesic[black]{0.37235}{-0.71426}{0.21315}{-0.27012}
\drawhypgeodesic[red, thick]{0.21315}{-0.27012}{0.64978}{0.39257}
\drawhypgeodesic[black]{0.64978}{0.39257}{0.19642}{0.78487}
\drawhypgeodesic[black]{0.19642}{0.78487}{-0.57328}{0.53708}
\drawhypgeodesic[black]{-0.57328}{0.53708}{-0.73680}{-0.19518}
\drawhypgeodesic[red, thick]{-0.73680}{-0.19518}{-0.12162}{-0.53496}
\drawhypgeodesic[black]{-0.12162}{-0.53496}{0.37235}{-0.71426}

\filldraw[fill=white] (0.37235,-0.71426) circle (0.05);

\end{scope}


\begin{scope}[shift={(1*2.3,2.3)}]

\drawhypgeodesic[black]{0.23887}{-0.64052}{0.63223}{-0.41092}
\drawhypgeodesic[black]{0.63223}{-0.41092}{0.70656}{0.06001}
\drawhypgeodesic[red, thick]{0.70656}{0.06001}{0.23025}{0.50535}
\drawhypgeodesic[black]{0.23025}{0.50535}{-0.48923}{0.48887}
\drawhypgeodesic[red, thick]{-0.48923}{0.48887}{-0.82041}{0.22969}
\drawhypgeodesic[black]{-0.82041}{0.22969}{-0.49828}{-0.23247}
\drawhypgeodesic[black]{-0.49828}{-0.23247}{0.23887}{-0.64052}

\filldraw[fill=white] (0.23887,-0.64052) circle (0.05);

\end{scope}


\begin{scope}[shift={(2*2.3,2.3)}]

\drawhypgeodesic[black]{0.16294}{-0.88229}{0.42314}{-0.70419}
\drawhypgeodesic[black]{0.42314}{-0.70419}{0.30843}{0.41590}
\drawhypgeodesic[red, thick]{0.30843}{0.41590}{-0.24948}{0.91594}
\drawhypgeodesic[black]{-0.24948}{0.91594}{-0.45640}{0.79381}
\drawhypgeodesic[red, thick]{-0.45640}{0.79381}{-0.12087}{0.31005}
\drawhypgeodesic[black]{-0.12087}{0.31005}{-0.06616}{-0.84473}
\drawhypgeodesic[black]{-0.06616}{-0.84473}{0.16294}{-0.88229}

\filldraw[fill=white] (0.16294,-0.88229) circle (0.05);

\end{scope}


\begin{scope}[shift={(3*2.3,2.3)}]

\drawhypgeodesic[red, thick]{0.63108}{-0.51695}{0.82318}{-0.25880}
\drawhypgeodesic[black]{0.82318}{-0.25880}{0.36958}{0.19925}
\drawhypgeodesic[red, thick]{0.36958}{0.19925}{-0.19898}{0.65481}
\drawhypgeodesic[black]{-0.19898}{0.65481}{-0.75833}{0.32539}
\drawhypgeodesic[blue, thick]{-0.75833}{0.32539}{-0.84325}{-0.04255}
\drawhypgeodesic[black]{-0.84325}{-0.04255}{-0.02327}{-0.36112}
\drawhypgeodesic[blue, thick]{-0.02327}{-0.36112}{0.63108}{-0.51695}

\filldraw[fill=white] (0.63108,-0.51695) circle (0.05);
  
\end{scope}


\begin{scope}[shift={(4*2.3,2.3)}]
\drawhypgeodesic[black]{-0.05539}{-0.75866}{0.58774}{-0.38939}
\drawhypgeodesic[blue, thick]{0.58774}{-0.38939}{0.60282}{0.31728}
\drawhypgeodesic[red, thick]{0.60282}{0.31728}{0.20990}{0.53864}
\drawhypgeodesic[black]{0.20990}{0.53864}{-0.32832}{0.62102}
\drawhypgeodesic[red, thick]{-0.32832}{0.62102}{-0.51384}{0.21519}
\drawhypgeodesic[blue, thick]{-0.51384}{0.21519}{-0.50291}{-0.54408}
\drawhypgeodesic[black]{-0.50291}{-0.54408}{-0.05539}{-0.75866}

\filldraw[fill=white] (-0.05539,-0.75866) circle (0.05);
  
\end{scope}


\begin{scope}[xshift=0*2.3 cm]

\drawhypgeodesic[red, thick]{0.45371}{-0.76088}{0.59077}{-0.45591}
\drawhypgeodesic[black]{0.59077}{-0.45591}{0.45399}{0.23260}
\drawhypgeodesic[red, thick]{0.45399}{0.23260}{-0.18403}{0.59066}
\drawhypgeodesic[black]{-0.18403}{0.59066}{-0.58910}{0.70827}
\drawhypgeodesic[blue, thick]{-0.58910}{0.70827}{-0.62911}{0.27208}
\drawhypgeodesic[black]{-0.62911}{0.27208}{-0.09620}{-0.58674}
\drawhypgeodesic[blue, thick]{-0.09620}{-0.58674}{0.45371}{-0.76088}

\filldraw[fill=white] (0.45371,-0.76088) circle (0.05);
  
\end{scope}


\begin{scope}[xshift=1*2.3 cm]

\drawhypgeodesic[red, thick]{-0.18747}{-0.60731}{0.57012}{-0.24422}
\drawhypgeodesic[blue, thick]{0.57012}{-0.24422}{0.78176}{0.22488}
\drawhypgeodesic[black]{0.78176}{0.22488}{0.49778}{0.47567}
\drawhypgeodesic[black]{0.49778}{0.47567}{-0.25948}{0.48846}
\drawhypgeodesic[red, thick]{-0.25948}{0.48846}{-0.79227}{0.02756}
\drawhypgeodesic[black]{-0.79227}{0.02756}{-0.61044}{-0.36504}
\drawhypgeodesic[blue, thick]{-0.61044}{-0.36504}{-0.18747}{-0.60731}

\filldraw[fill=white] (-0.18747,-0.60731) circle (0.05);
  
\end{scope}


\begin{scope}[xshift=2*2.3 cm]

\drawhypgeodesic[red, thick]{0.01376}{-0.58397}{0.54117}{-0.42546}
\drawhypgeodesic[black]{0.54117}{-0.42546}{0.80503}{-0.05988}
\drawhypgeodesic[blue, thick]{0.80503}{-0.05988}{0.46029}{0.37917}
\drawhypgeodesic[black]{0.46029}{0.37917}{-0.34724}{0.52215}
\drawhypgeodesic[blue, thick]{-0.34724}{0.52215}{-0.78787}{0.26461}
\drawhypgeodesic[red, thick]{-0.78787}{0.26461}{-0.68514}{-0.09663}
\drawhypgeodesic[black]{-0.68514}{-0.09663}{0.01376}{-0.58397}

\filldraw[fill=white] (0.01376,-0.58397) circle (0.05);
   
\end{scope}


\begin{scope}[xshift=3*2.3 cm]

\drawhypgeodesic[red, thick]{0.12963}{-0.79873}{0.53788}{-0.57426}
\drawhypgeodesic[black]{0.53788}{-0.57426}{0.64261}{0.09743}
\drawhypgeodesic[blue, thick]{0.64261}{0.09743}{-0.15898}{0.61576}
\drawhypgeodesic[black]{-0.15898}{0.61576}{-0.34770}{0.86621}
\drawhypgeodesic[blue, thick]{-0.34770}{0.86621}{-0.43802}{0.40678}
\drawhypgeodesic[black]{-0.43802}{0.40678}{-0.36541}{-0.61318}
\drawhypgeodesic[red, thick]{-0.36541}{-0.61318}{0.12963}{-0.79873}

\filldraw[fill=white] (0.12963,-0.79873) circle (0.05);
  
\end{scope}


\begin{scope}[xshift=4*2.3 cm]

\drawhypgeodesic[black]{-0.13226}{-0.52956}{0.45291}{-0.56173}
\drawhypgeodesic[red, thick]{0.45291}{-0.56173}{0.61787}{-0.10596}
\drawhypgeodesic[black]{0.61787}{-0.10596}{0.47514}{0.54988}
\drawhypgeodesic[blue, thick]{0.47514}{0.54988}{-0.07206}{0.68331}
\drawhypgeodesic[black]{-0.07206}{0.68331}{-0.64232}{0.20378}
\drawhypgeodesic[red, thick]{-0.64232}{0.20378}{-0.69928}{-0.23971}
\drawhypgeodesic[blue, thick]{-0.69928}{-0.23971}{-0.13226}{-0.52956}

\filldraw[fill=white] (-0.13226,-0.52956) circle (0.05);
   
\end{scope}


\begin{scope}[xshift=5*2.3 cm]

\drawhypgeodesic[red, thick]{0.01787}{-0.77854}{0.50439}{-0.51112}
\drawhypgeodesic[black]{0.50439}{-0.51112}{0.46113}{0.30445}
\drawhypgeodesic[blue, thick]{0.46113}{0.30445}{0.21227}{0.72981}
\drawhypgeodesic[black]{0.21227}{0.72981}{-0.25251}{0.66118}
\drawhypgeodesic[blue, thick]{-0.25251}{0.66118}{-0.46439}{0.14063}
\drawhypgeodesic[black]{-0.46439}{0.14063}{-0.47875}{-0.54641}
\drawhypgeodesic[red, thick]{-0.47875}{-0.54641}{0.01787}{-0.77854}

\filldraw[fill=white] (0.01787,-0.77854) circle (0.05);
  
\end{scope}


\begin{scope}[xshift=6*2.3 cm]

\drawhypgeodesic[black]{-0.24451}{-0.56244}{0.45927}{-0.38593}
\drawhypgeodesic[red, thick]{0.45927}{-0.38593}{0.74978}{0.00048}
\drawhypgeodesic[black]{0.74978}{0.00048}{0.47632}{0.56682}
\drawhypgeodesic[blue, thick]{0.47632}{0.56682}{-0.07706}{0.59197}
\drawhypgeodesic[black]{-0.07706}{0.59197}{-0.66266}{0.10922}
\drawhypgeodesic[red, thick]{-0.66266}{0.10922}{-0.70113}{-0.32013}
\drawhypgeodesic[blue, thick]{-0.70113}{-0.32013}{-0.24451}{-0.56244}

\filldraw[fill=white] (-0.24451,-0.56244) circle (0.05);
  
\end{scope}


\begin{scope}[xshift=7*2.3 cm]

\drawhypgeodesic[black]{-0.13187}{-0.39203}{0.58733}{-0.56786}
\drawhypgeodesic[red, thick]{0.58733}{-0.56786}{0.77888}{-0.32646}
\drawhypgeodesic[black]{0.77888}{-0.32646}{0.41286}{0.15757}
\drawhypgeodesic[blue, thick]{0.41286}{0.15757}{-0.29653}{0.70964}
\drawhypgeodesic[black]{-0.29653}{0.70964}{-0.64901}{0.51714}
\drawhypgeodesic[blue, thick]{-0.64901}{0.51714}{-0.70166}{-0.09801}
\drawhypgeodesic[red, thick]{-0.70166}{-0.09801}{-0.13187}{-0.39203}

\filldraw[fill=white] (-0.13187,-0.39203) circle (0.05);
  
\end{scope}

\end{tikzpicture}
\caption{Geometrical realisations for tilings in Figure \ref{T2A}. In the first row, the two tiles have the same orientation. In the second row, the two tiles have different orientation.  $\circ$ is the location of the corner $0$.} 
\label{T2Ageom}
\end{figure}

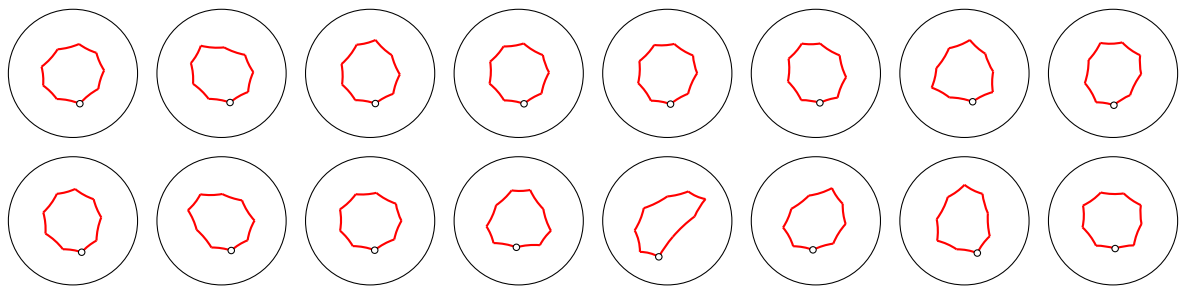
\begin{figure}[htp]
\centering
\begin{tikzpicture}[>=latex,scale=0.85]

\foreach \x in {0,...,7}
\foreach \y in {0,1}
\draw
	(2.3*\x, -2.3*\y) circle (1);

\begin{scope}[shift={(0,0)}]

\drawhypgeodesic[red, thick]{0.10659}{-0.47319}{0.39284}{-0.24247}
\drawhypgeodesic[red, thick]{0.39284}{-0.24247}{0.48294}{0.04696}
\drawhypgeodesic[red, thick]{0.48294}{0.04696}{0.36286}{0.32076}
\drawhypgeodesic[red, thick]{0.36286}{0.32076}{0.09354}{0.45798}
\drawhypgeodesic[red, thick]{0.09354}{0.45798}{-0.23847}{0.37564}
\drawhypgeodesic[red, thick]{-0.23847}{0.37564}{-0.48651}{0.10350}
\drawhypgeodesic[red, thick]{-0.48651}{0.10350}{-0.46249}{-0.18919}
\drawhypgeodesic[red, thick]{-0.46249}{-0.18919}{-0.25130}{-0.39998}
\drawhypgeodesic[red, thick]{-0.25130}{-0.39998}{0.10659}{-0.47319}

\filldraw[fill=white] (0.10659,-0.47319) circle (0.05);

\end{scope}

\begin{scope}[shift={(2.3*1,0)}]
\drawhypgeodesic[red, thick]{0.13152}{-0.45462}{0.41183}{-0.29048}
\drawhypgeodesic[red, thick]{0.41183}{-0.29048}{0.49198}{0.02029}
\drawhypgeodesic[red, thick]{0.49198}{0.02029}{0.36393}{0.28897}
\drawhypgeodesic[red, thick]{0.36393}{0.28897}{0.03890}{0.40541}
\drawhypgeodesic[red, thick]{0.03890}{0.40541}{-0.32077}{0.43139}
\drawhypgeodesic[red, thick]{-0.32077}{0.43139}{-0.47488}{0.16071}
\drawhypgeodesic[red, thick]{-0.47488}{0.16071}{-0.44213}{-0.16343}
\drawhypgeodesic[red, thick]{-0.44213}{-0.16343}{-0.20038}{-0.39825}
\drawhypgeodesic[red, thick]{-0.20038}{-0.39825}{0.13152}{-0.45462}

\filldraw[fill=white] (0.13152,-0.45462) circle (0.05);
\end{scope}

\begin{scope}[shift={(2.3*2,0)}]
\drawhypgeodesic[red, thick]{0.08091}{-0.47044}{0.35798}{-0.31901}
\drawhypgeodesic[red, thick]{0.35798}{-0.31901}{0.46028}{-0.01692}
\drawhypgeodesic[red, thick]{0.46028}{-0.01692}{0.33517}{0.28459}
\drawhypgeodesic[red, thick]{0.33517}{0.28459}{0.08390}{0.52155}
\drawhypgeodesic[red, thick]{0.08390}{0.52155}{-0.20920}{0.44764}
\drawhypgeodesic[red, thick]{-0.20920}{0.44764}{-0.44307}{0.14524}
\drawhypgeodesic[red, thick]{-0.44307}{0.14524}{-0.43513}{-0.17462}
\drawhypgeodesic[red, thick]{-0.43513}{-0.17462}{-0.23084}{-0.41803}
\drawhypgeodesic[red, thick]{-0.23084}{-0.41803}{0.08091}{-0.47044}

\filldraw[fill=white] (0.08091,-0.47044) circle (0.05);
\end{scope}

\begin{scope}[shift={(2.3*3,0)}]
\drawhypgeodesic[red, thick]{0.08226}{-0.47468}{0.35364}{-0.31422}
\drawhypgeodesic[red, thick]{0.35364}{-0.31422}{0.47000}{0.01386}
\drawhypgeodesic[red, thick]{0.47000}{0.01386}{0.37385}{0.28518}
\drawhypgeodesic[red, thick]{0.37385}{0.28518}{0.07687}{0.46678}
\drawhypgeodesic[red, thick]{0.07687}{0.46678}{-0.23499}{0.42056}
\drawhypgeodesic[red, thick]{-0.23499}{0.42056}{-0.44212}{0.18442}
\drawhypgeodesic[red, thick]{-0.44212}{0.18442}{-0.45217}{-0.16024}
\drawhypgeodesic[red, thick]{-0.45217}{-0.16024}{-0.22734}{-0.42166}
\drawhypgeodesic[red, thick]{-0.22734}{-0.42166}{0.08226}{-0.47468}

\filldraw[fill=white] (0.08226,-0.47468) circle (0.05);
\end{scope}

\begin{scope}[shift={(2.3*4,0)}]
\drawhypgeodesic[red, thick]{0.04990}{-0.47947}{0.32196}{-0.31828}
\drawhypgeodesic[red, thick]{0.32196}{-0.31828}{0.46249}{0.00314}
\drawhypgeodesic[red, thick]{0.46249}{0.00314}{0.39852}{0.27956}
\drawhypgeodesic[red, thick]{0.39852}{0.27956}{0.10814}{0.45963}
\drawhypgeodesic[red, thick]{0.10814}{0.45963}{-0.20297}{0.44329}
\drawhypgeodesic[red, thick]{-0.20297}{0.44329}{-0.43284}{0.19096}
\drawhypgeodesic[red, thick]{-0.43284}{0.19096}{-0.45178}{-0.15217}
\drawhypgeodesic[red, thick]{-0.45178}{-0.15217}{-0.25343}{-0.42666}
\drawhypgeodesic[red, thick]{-0.25343}{-0.42666}{0.04990}{-0.47947}

\filldraw[fill=white] (0.04990,-0.47947) circle (0.05);
\end{scope}

\begin{scope}[shift={(2.3*5,0)}]
\drawhypgeodesic[red, thick]{0.06210}{-0.46008}{0.33683}{-0.37886}
\drawhypgeodesic[red, thick]{0.33683}{-0.37886}{0.46865}{-0.05956}
\drawhypgeodesic[red, thick]{0.46865}{-0.05956}{0.36859}{0.28115}
\drawhypgeodesic[red, thick]{0.36859}{0.28115}{0.06549}{0.46664}
\drawhypgeodesic[red, thick]{0.06549}{0.46664}{-0.21846}{0.46494}
\drawhypgeodesic[red, thick]{-0.21846}{0.46494}{-0.42304}{0.22956}
\drawhypgeodesic[red, thick]{-0.42304}{0.22956}{-0.43299}{-0.12604}
\drawhypgeodesic[red, thick]{-0.43299}{-0.12604}{-0.22718}{-0.41776}
\drawhypgeodesic[red, thick]{-0.22718}{-0.41776}{0.06210}{-0.46008}

\filldraw[fill=white] (0.06210,-0.46008) circle (0.05);
\end{scope}

\begin{scope}[shift={(2.3*6,0)}]
\drawhypgeodesic[red, thick]{0.12670}{-0.44279}{0.44006}{-0.29166}
\drawhypgeodesic[red, thick]{0.44006}{-0.29166}{0.44470}{0.02827}
\drawhypgeodesic[red, thick]{0.44470}{0.02827}{0.32224}{0.30776}
\drawhypgeodesic[red, thick]{0.32224}{0.30776}{0.08390}{0.52123}
\drawhypgeodesic[red, thick]{0.08390}{0.52123}{-0.23963}{0.39332}
\drawhypgeodesic[red, thick]{-0.23963}{0.39332}{-0.43522}{0.08247}
\drawhypgeodesic[red, thick]{-0.43522}{0.08247}{-0.50837}{-0.22274}
\drawhypgeodesic[red, thick]{-0.50837}{-0.22274}{-0.23441}{-0.37586}
\drawhypgeodesic[red, thick]{-0.23441}{-0.37586}{0.12670}{-0.44279}

\filldraw[fill=white] (0.12670,-0.44279) circle (0.05);
\end{scope}

\begin{scope}[shift={(2.3*7,0)}]
\drawhypgeodesic[red, thick]{0.01327}{-0.49703}{0.26417}{-0.34010}
\drawhypgeodesic[red, thick]{0.26417}{-0.34010}{0.43528}{-0.00963}
\drawhypgeodesic[red, thick]{0.43528}{-0.00963}{0.41934}{0.28337}
\drawhypgeodesic[red, thick]{0.41934}{0.28337}{0.14258}{0.48305}
\drawhypgeodesic[red, thick]{0.14258}{0.48305}{-0.19659}{0.47563}
\drawhypgeodesic[red, thick]{-0.19659}{0.47563}{-0.37865}{0.20896}
\drawhypgeodesic[red, thick]{-0.37865}{0.20896}{-0.43678}{-0.15251}
\drawhypgeodesic[red, thick]{-0.43678}{-0.15251}{-0.26261}{-0.45173}
\drawhypgeodesic[red, thick]{-0.26261}{-0.45173}{0.01327}{-0.49703}

\filldraw[fill=white] (0.01327,-0.497036) circle (0.05);
\end{scope}

\begin{scope}[shift={(0,-2.3)}]
\drawhypgeodesic[red, thick]{0.13476}{-0.48966}{0.36430}{-0.30803}
\drawhypgeodesic[red, thick]{0.36430}{-0.30803}{0.43855}{0.05402}
\drawhypgeodesic[red, thick]{0.43855}{0.05402}{0.32009}{0.33730}
\drawhypgeodesic[red, thick]{0.32009}{0.33730}{0.03323}{0.49667}
\drawhypgeodesic[red, thick]{0.03323}{0.49667}{-0.24931}{0.42156}
\drawhypgeodesic[red, thick]{-0.24931}{0.42156}{-0.45976}{0.13074}
\drawhypgeodesic[red, thick]{-0.45976}{0.13074}{-0.42631}{-0.19885}
\drawhypgeodesic[red, thick]{-0.42631}{-0.19885}{-0.15556}{-0.44375}
\drawhypgeodesic[red, thick]{-0.15556}{-0.44375}{0.13476}{-0.48966}

\filldraw[fill=white] (0.13476,-0.48966) circle (0.05);
\end{scope}

\begin{scope}[shift={(2.3*1,-2.3)}]
\drawhypgeodesic[red, thick]{0.15016}{-0.46306}{0.40406}{-0.30402}
\drawhypgeodesic[red, thick]{0.40406}{-0.30402}{0.51103}{0.00392}
\drawhypgeodesic[red, thick]{0.51103}{0.00392}{0.31715}{0.31929}
\drawhypgeodesic[red, thick]{0.31715}{0.31929}{0.01031}{0.41838}
\drawhypgeodesic[red, thick]{0.01031}{0.41838}{-0.32861}{0.41382}
\drawhypgeodesic[red, thick]{-0.32861}{0.41382}{-0.52087}{0.16761}
\drawhypgeodesic[red, thick]{-0.52087}{0.16761}{-0.38592}{-0.13959}
\drawhypgeodesic[red, thick]{-0.38592}{-0.13959}{-0.15730}{-0.41635}
\drawhypgeodesic[red, thick]{-0.15730}{-0.41635}{0.15016}{-0.46306}

\filldraw[fill=white] (0.15016,-0.46306) circle (0.05);
\end{scope}

\begin{scope}[shift={(2.3*2,-2.3)}]
\drawhypgeodesic[red, thick]{0.07083}{-0.45839}{0.37383}{-0.29082}
\drawhypgeodesic[red, thick]{0.37383}{-0.29082}{0.48561}{0.00011}
\drawhypgeodesic[red, thick]{0.48561}{0.00011}{0.39936}{0.26739}
\drawhypgeodesic[red, thick]{0.39936}{0.26739}{0.09520}{0.43670}
\drawhypgeodesic[red, thick]{0.09520}{0.43670}{-0.22343}{0.41503}
\drawhypgeodesic[red, thick]{-0.22343}{0.41503}{-0.46710}{0.17959}
\drawhypgeodesic[red, thick]{-0.46710}{0.17959}{-0.46562}{-0.15515}
\drawhypgeodesic[red, thick]{-0.46562}{-0.15515}{-0.26867}{-0.39446}
\drawhypgeodesic[red, thick]{-0.26867}{-0.39446}{0.07083}{-0.45839}

\filldraw[fill=white] (0.07083,-0.45839) circle (0.05);
\end{scope}

\begin{scope}[shift={(2.3*3,-2.3)}]
\drawhypgeodesic[red, thick]{-0.03588}{-0.41312}{0.32822}{-0.37985}
\drawhypgeodesic[red, thick]{0.32822}{-0.37985}{0.49942}{-0.15513}
\drawhypgeodesic[red, thick]{0.49942}{-0.15513}{0.38167}{0.18370}
\drawhypgeodesic[red, thick]{0.38167}{0.18370}{0.17714}{0.48146}
\drawhypgeodesic[red, thick]{0.17714}{0.48146}{-0.11000}{0.47497}
\drawhypgeodesic[red, thick]{-0.11000}{0.47497}{-0.35065}{0.24257}
\drawhypgeodesic[red, thick]{-0.35065}{0.24257}{-0.49963}{-0.08752}
\drawhypgeodesic[red, thick]{-0.49963}{-0.08752}{-0.39029}{-0.34707}
\drawhypgeodesic[red, thick]{-0.39029}{-0.34707}{-0.03588}{-0.41312}

\filldraw[fill=white] (-0.03588,-0.4131) circle (0.05);
\end{scope}

\begin{scope}[shift={(2.3*4,-2.3)}]
\drawhypgeodesic[red, thick]{-0.13144}{-0.56184}{0.08606}{-0.24892}
\drawhypgeodesic[red, thick]{0.08606}{-0.24892}{0.42352}{0.06798}
\drawhypgeodesic[red, thick]{0.42352}{0.06798}{0.59159}{0.33966}
\drawhypgeodesic[red, thick]{0.59159}{0.33966}{0.32439}{0.45878}
\drawhypgeodesic[red, thick]{0.32439}{0.45878}{-0.00058}{0.38033}
\drawhypgeodesic[red, thick]{-0.00058}{0.38033}{-0.36433}{0.19753}
\drawhypgeodesic[red, thick]{-0.36433}{0.19753}{-0.50549}{-0.14892}
\drawhypgeodesic[red, thick]{-0.50549}{-0.14892}{-0.42371}{-0.48461}
\drawhypgeodesic[red, thick]{-0.42371}{-0.48461}{-0.13144}{-0.56184}

\filldraw[fill=white] (-0.13144,-0.56184) circle (0.05);
\end{scope}

\begin{scope}[shift={(2.3*5,-2.3)}]
\drawhypgeodesic[red, thick]{-0.04503}{-0.45305}{0.26926}{-0.35560}
\drawhypgeodesic[red, thick]{0.26926}{-0.35560}{0.46097}{-0.04824}
\drawhypgeodesic[red, thick]{0.46097}{-0.04824}{0.41557}{0.26669}
\drawhypgeodesic[red, thick]{0.41557}{0.26669}{0.24905}{0.51076}
\drawhypgeodesic[red, thick]{0.24905}{0.51076}{-0.10886}{0.39384}
\drawhypgeodesic[red, thick]{-0.10886}{0.39384}{-0.38069}{0.18751}
\drawhypgeodesic[red, thick]{-0.38069}{0.18751}{-0.50863}{-0.10527}
\drawhypgeodesic[red, thick]{-0.50863}{-0.10527}{-0.35164}{-0.39664}
\drawhypgeodesic[red, thick]{-0.35164}{-0.39664}{-0.04503}{-0.45305}

\filldraw[fill=white] (-0.04503,-0.45305) circle (0.05);
\end{scope}

\begin{scope}[shift={(2.3*6,-2.3)}]
\drawhypgeodesic[red, thick]{0.19911}{-0.50272}{0.39195}{-0.23947}
\drawhypgeodesic[red, thick]{0.39195}{-0.23947}{0.36270}{0.11793}
\drawhypgeodesic[red, thick]{0.36270}{0.11793}{0.27903}{0.42164}
\drawhypgeodesic[red, thick]{0.27903}{0.42164}{0.00072}{0.56193}
\drawhypgeodesic[red, thick]{0.00072}{0.56193}{-0.25852}{0.37046}
\drawhypgeodesic[red, thick]{-0.25852}{0.37046}{-0.43019}{0.02506}
\drawhypgeodesic[red, thick]{-0.43019}{0.02506}{-0.42311}{-0.31034}
\drawhypgeodesic[red, thick]{-0.42311}{-0.31034}{-0.12168}{-0.44449}
\drawhypgeodesic[red, thick]{-0.12168}{-0.44449}{0.19911}{-0.50272}

\filldraw[fill=white] (0.19911,-0.50272) circle (0.05);
\end{scope}

\begin{scope}[shift={(2.3*7,-2.3)}]
\drawhypgeodesic[red, thick]{0.03265}{-0.43255}{0.32269}{-0.38292}
\drawhypgeodesic[red, thick]{0.32269}{-0.38292}{0.44102}{-0.04855}
\drawhypgeodesic[red, thick]{0.44102}{-0.04855}{0.43051}{0.24499}
\drawhypgeodesic[red, thick]{0.43051}{0.24499}{0.16554}{0.42774}
\drawhypgeodesic[red, thick]{0.16554}{0.42774}{-0.19035}{0.44192}
\drawhypgeodesic[red, thick]{-0.19035}{0.44192}{-0.46518}{0.23896}
\drawhypgeodesic[red, thick]{-0.46518}{0.23896}{-0.47270}{-0.10208}
\drawhypgeodesic[red, thick]{-0.47270}{-0.10208}{-0.26419}{-0.38750}
\drawhypgeodesic[red, thick]{-0.26419}{-0.38750}{0.03265}{-0.43255}

\filldraw[fill=white] (0.03265,-0.43255) circle (0.05);
\end{scope}

\end{tikzpicture}
\caption{Geometrical realisations for tilings in Figure \ref{P2C}.} 
\label{P2Cgeom}
\end{figure}

\section{Distinct Edge Lengths}
\label{edge}

In an edge-to-edge tilings of a surface of Euler number $\chi<0$ by $f$ congruent $n$-gons, $n\ge 7$, we discuss the condition that all $n$ edges have distinct edge length. The assumptions will not be repeated in the subsequent propositions.

Combinatorially, all edges have distinct edge length if and only if the edge pairs in the multiple planar diagram are $\bar{i}_p\bar{i}_q$. An immediate consequence is that $f$ must be even. 

\begin{proposition}
If the $n$-gon has $n$ distinct edge lengths, then $f$ is even, and there is no degree $3$ vertex.
\end{proposition}

\begin{proof}
The condition of no degree $3$ vertex is already proved in \cite[Proposition 5]{gsy} and \cite[Lemma 9]{wy1}. We reproduce the proof here.

In Figure \ref{elength}, we consider a degree 3 vertex $\bullet$ and three tiles $t,t_1,t_2$ around it. Suppose all edge lengths are distinct, and three edges of $t$ have lengths $a,b,c$. Since the edge length $x$ is adjacent to $b$ in $t_1$, and $t_1,t_2$ are congruent, and all edges have distinct lengths, we know $x=a$ or $c$. If $x=a$, then $t_2$ has two edges of length $a$. If $x=c$, then $a$ and $c$ are adjacent in $t_2$, contradicting the assumption that $t,t_2$ are congruent. 
\end{proof}

\begin{figure}[htp]
\centering
\begin{tikzpicture}[>=latex,scale=1]

\foreach \a in {0,1,2}
\draw[rotate=120*\a]
	(0,0) -- (1,0);

\draw
	(1,0) -- ++(60:1);
		
\fill
	(0,0) circle (0.05);

\begin{scope}[font=\footnotesize]

\node at (-0.13,0.5) {$a$};
\node at (0.5,0.15) {$b$};
\node at (1.13,0.5) {$c$};
\node at (-0.13,-0.5) {$x$};

\node at (0.5,0.8) {$t$};
\node at (-0.7,0) {$t_2$};
\node at (0.5,-0.8) {$t_1$};

\end{scope}
	
\end{tikzpicture}
\caption{Degree $3$ vertex implies some edge lengths are equal.} 
\label{elength}
\end{figure}
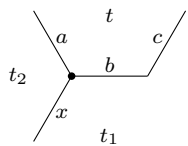

In \cite{lwwy1}, we showed that $f$ reaches the maximal value $\frac{-6\chi}{n-6}$ if and only if all vertices have degree 3. Therefore we expect that, if $f$ is close to the maximal value, then the tiling has degree 3 vertex. We assume $f=\frac{-6\chi}{(1+\epsilon)(n-6)}$, and find how small $\epsilon>0$ needs to be in order for the tiling to necessarily have degree 3 vertex. The idea for the subsequent argument comes from \cite[Section 3.2]{cly5}.

In the usual equalities
\begin{align}
v-e+f &=\chi, \nonumber \\
v &= v_3+v_4+v_5+\cdots, \nonumber \\
nf=2e &= 3v_3+4v_4+5v_5+\cdots
=3v+v_4+2v_5+\cdots, \label{eq3}
\end{align}
we candel $v_3,e,v$ and get
\[
v_4+2v_5+\cdots=nf-3\chi-3e+3f
=\frac{6-n}{2}f-3\chi
=\frac{n-6}{2}\epsilon f.
\]
This is the same as
\[
nf=\frac{2n}{\epsilon(n-6)}\sum_{k\ge 4}(k-3)v_k.
\]
We compare the coefficient $\frac{2n}{\epsilon(n-6)}(k-3)$ of $v_k$ in the equality above with the coefficient $k$ of $v_k$ in \eqref{eq3}. If $\epsilon\le \frac{n}{2(n-6)}$, then we get $\frac{2n}{\epsilon(n-6)}(4-3)\ge 4$, and $\frac{2n}{\epsilon(n-6)}(k-3)>k$ for $k>4$. Then to get $v_3=0$, we need $\epsilon= \frac{n}{2(n-6)}$ and all $v_k=0$ for $k\ge 5$. Since $\epsilon\ge  \frac{n}{2(n-6)}$ is the same as $f\ge \frac{-4\chi}{n-4}$, we conclude the following. 

\begin{proposition}
If $f>\frac{-4\chi}{n-4}$, then some edges in the $n$-gon have the same edge length. If $f=\frac{-4\chi}{n-4}$ and all edges have distinct lengths, then all vertices have degree $4$.
\end{proposition}

The critical value $\frac{-4\chi}{n-4}$ divides the range $\frac{-2\chi}{n-2}<f\le \frac{-6\chi}{n-6}$ for $f$ into two parts. For all edges to have distinct lengths, $f$ needs to be in the lower part of the range.

For two tile tilings, we have a complete description of all tilings such that all edges have distinct lengths. The double planar diagram consists of $(\bar{i}_1\bar{i}_2)_{\sigma}$ (the $i$-th edges in tiles 1 and 2, {\em not} $i_1$-th edge and $i_2$-th edge). By the conditions for multiple planar diagram stated after Proposition \ref{avs2}, the only requirement is that $(\bar{i}_1\bar{i}_2)_-$ and $(\overline{i+1}_1\overline{i+1}_2)_-$ cannot both be edge pairs. This means twisted pairs are not adjacent. In other words, the twisted edge pairs are $(\bar{i}_1\bar{i}_2)_-$, for $i=k_1,k_2,\dots,k_{\tau}\in {\bb Z}_n$ satisfying ($k_{\tau+1}=k_1+n$)
\[
0\le k_1<k_2<\dots<k_{\tau}<n,\quad 
k_{p+1}-k_p\ge 2.
\]
Between the twisted pairs $(\overline{k_p}_1\overline{k_p}_2)_-$ and $(\overline{k_{p+1}}_1\overline{k_{p+1}}_2)_-$ is a sequence of opposing pairs $(\bar{i}_1\bar{i}_2)_+$ for $i=k_p+1,k_p+2,\dots k_{p+1}-1$. In Figure \ref{distinct}, we follow the lower right of Figure \ref{p2v}, and find that these opposing pairs form a vertex of degree $2(k_{p+1}-k_p)$, and all vertices are gives by these opposing pair sequences. For $\tau>0$, the number of vertices in the tiling is $\tau$, and the Euler number $\chi=\tau-n+2$, and the surface is not orientable. For $\tau=0$, the number of vertices is 1, and the Euler number $\chi=1-n+2=3-n$, and the surface is orientable. In particular, $n$ is odd in this case.

\begin{figure}[htp]
\centering
\begin{tikzpicture}[>=latex,scale=1]

\foreach \b in {1,-1}
{
\foreach \a in {0,1,4}
\draw[gray!50, very thick]
	(\a,0.93*\b) -- (1+\a,-0.93*\b);

\draw[gray!50, very thick]
	(2,0.93*\b) -- ++(0.2,-0.4*\b)
	(4,0.93*\b) -- ++(-0.2,-0.4*\b);

}

\foreach \a in {-1,0,5,6}
\draw[gray!50, very thick]
	(\a,0.93) -- (\a,-0.93);

\foreach \a in {-1,5}
\draw[dash pattern=on 2pt off 1pt]
	(0.5+\a,1) -- ++(0,-2);

\foreach \a in {0,1,4}
\draw
	(0.5+\a,1) -- ++(0,-2);

\foreach \b in {1,-1}
{
\draw[dashed, yscale=\b]
	(-1,1) to[out=180, in=-50] 
	(-1.5,1.3) to[out=130, in=50]
	(6.5,1.3) to[out=230, in=0] 
	(6,1)
	(2,1) -- (4,1);

\draw
	(-1,\b) -- (2,\b)
	(4,\b) -- (6,\b);
	
\draw[shift={(2.5, 2*\b)}, yscale=\b, ->]
	(120:0.3) arc (120:420:0.3);

\foreach \a in {-1,0,1,2,4,5,6}
\fill
	(\a,\b) circle (0.05);
	
\foreach \a in {1,2}
\node at (-0.5+\a,1.2*\b) {\scriptsize $\overline{k_p\!+\!\a}$};

\node at (4.5,1.2*\b) {\scriptsize $\overline{k_{p+1}\!-\!1}$};
\node at (5.5,1.2*\b) {\scriptsize $\overline{k_{p+1}}$};
\node at (-0.5,1.2*\b) {\scriptsize $\overline{k_p}$};
}
			
\end{tikzpicture}
\caption{Two tile tilings with all distinct edge lengths.} 
\label{distinct}
\end{figure}
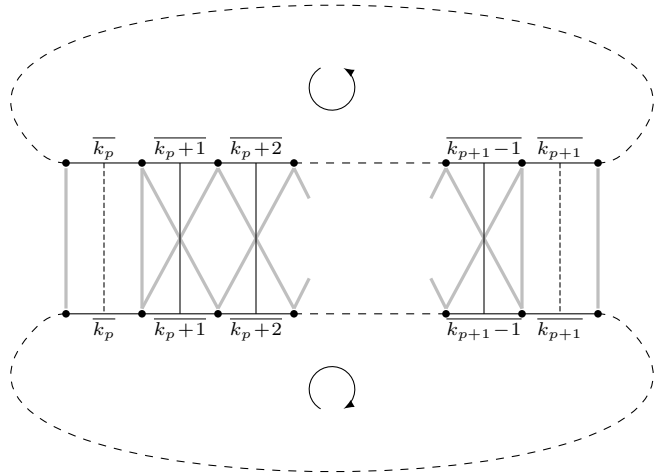

Since twisted pairs are not adjacent, the maximal number of twisted pairs is the integer part of $\frac{n}{2}$. Then we get the following result.

\begin{proposition}
In a tiling of surface by two congruent $n$-gons, all $n$ edges can have distinct lengths if and only if all edge pairs are $(\bar{i}_1\bar{i}_2)_{\sigma}$, and twisted edge pairs are not adjacent. The following are the surfaces admitting tilings by two congruent $n$-gons, such that all $n$ edges have distinct lengths:
\begin{itemize}
\item $n=2m$ is even: surface $g{\bb P}^2$, with $m\le g\le 2m-1$.
\item $n=2m-1$ is odd: surface $g{\bb P}^2$, with $m\le g\le 2m-2$, and surface $(m-1){\bb T}^2$.
\end{itemize}
\end{proposition}

\end{document}